\documentclass[11pt]{amsart}
\usepackage{geometry}
\usepackage{graphicx}
\usepackage{amssymb}
\usepackage{amsmath, mathrsfs, stmaryrd, appendix, tensor, comment, soul, ulem}

\usepackage{amsthm, amssymb}

\newcommand{\na}{\nabla}
\newcommand{\mxtr}{\mbox{tr}}

\newcommand{\rw}{\rightarrow}
\newcommand{\vphi}{\varphi}
\newcommand{\lt}{\left}
\newcommand{\rt}{\right}

\usepackage{color}
\newtheorem{theorem}{Theorem}
\newtheorem*{Theorem A}{Theorem A}
\newtheorem*{Theorem B}{Theorem B}
\newtheorem*{Theorem B'}{Theorem B'}
\newtheorem*{Theorem C}{Theorem C}
\newtheorem*{Corollary C}{Corollary C}
\newtheorem*{Theorem D}{Theorem D}
\newtheorem*{Theorem E}{Theorem E}
\newtheorem*{Theorem F}{Theorem F}

\newtheorem{lemma}[theorem]{Lemma}
\newtheorem{lem}[theorem]{Lemma}
\newtheorem{corollary}[theorem]{Corollary}
\newtheorem{proposition}[theorem]{Proposition}

\theoremstyle{remark}
\newtheorem{remark}[theorem]{Remark}

\theoremstyle{definition}
\newtheorem{definition}[theorem]{Definition}
\newtheorem{assumption}[theorem]{Assumption}
\newcommand{\R}{\mathbb R}
\numberwithin{theorem}{section}
\numberwithin{equation}{section}

\newcommand{\udl}{\underline}

\newcommand{\Lb}{\udl{L}}
\newcommand{\Tb}{\udl{T}}
\newcommand{\chib}{\udl{\chi}}
\newcommand{\omegab}{\underline{\omega}}
\newcommand{\pl}{\partial}

\begin{document}
\title{Minkowski formulae and Alexandrov theorems in spacetime}
\author{Mu-Tao Wang, Ye-Kai Wang, and Xiangwen Zhang}
\address{Mu-Tao Wang\\
Department of Mathematics\\
Columbia University, U.S.} \email{mtwang@math.columbia.edu}

\address{Ye-Kai Wang\\
Department of Mathematics\\
Columbia University, U.S.} \email{yw2293@columbia.edu}

\address{Xiangwen Zhang\\
Department of Mathematics\\
Columbia University, U.S.\\
Current: Department of Mathematics, University of California, Irvine} \email{xzhang@math.uci.edu}

\begin{abstract}
The classical Minkowski formula is extended to spacelike codimension-two submanifolds in spacetimes which admit ``hidden symmetry" from conformal Killing-Yano two-forms. As an application, we obtain an Alexandrov type theorem for spacelike codimension-two submanifolds in a static spherically symmetric spacetime: a codimension-two submanifold with constant normalized null expansion (null mean curvature) must lie in a shear-free (umbilical) null hypersurface. These results are generalized for higher order curvature invariants. In particular, the notion of {\it mixed higher order mean curvature} is introduced to highlight the special null geometry of 
the submanifold. Finally, Alexandrov type theorems are established for spacelike submanifolds with constant mixed higher order mean curvature, which are generalizations of hypersurfaces of constant Weingarten curvature in the Euclidean space. 
\end{abstract}

\thanks{M.-T. Wang is supported by NSF grant DMS-1105483 and DMS-1405152. X.W. Zhang is supported by NSF grant DMS-1308136. This work was partially supported by a grant from the Simons Foundation (\#305519 to Mu-Tao Wang).}

\thanks{}
\maketitle

\section{Introduction}
\par
For a smooth closed oriented hypersurface $X: \Sigma \rw \mathbb{R}^n,$ the $k$-th Minkowski formula reads
\begin{equation}\label{Minkowski.formula}
(n-k) \int_\Sigma \sigma_{k-1} d\mu = k\int_\Sigma \sigma_k \langle X,\nu \rangle d\mu
\end{equation} where $\sigma_k$ is the $k$-th elementary symmetric function of the principal curvatures and $\nu$ is the outward unit normal vector field of $\Sigma$. 
\eqref{Minkowski.formula} was proved by Minkowski \cite{Minkowski} for convex hypersurfaces and generalized by Hsiung \cite{Hsiung54} to all hypersurfaces stated above. There are also generalizations for various ambient spaces and higher codimensional  submanifolds \cite{GuanLi, Kwong13, Strubing84}.

The Minkowski formula is closely related to the conformal symmetry of the ambient space. Indeed, the position vector $X$ in \eqref{Minkowski.formula} should be regarded as the restriction of the conformal Killing vector field $r \frac{\pl}{\pl r}$ on the hypersurface $\Sigma.$ In this paper, we make use of the conformal Killing-Yano two-forms (see Definition \ref{C-K}) and discover several new Minkowski formulae for spacelike codimension-two submanifolds in Lorentzian manifolds. Unlike conformal Killing vector fields, conformal Killing-Yano two-forms are the so-called ``hidden symmetry" which may not correspond to any continuous symmetry of the ambient space. 

In the introduction, we specialize our discussion to the Schwarzschild spacetime and spacetimes of constant curvature. Several theorems proved in this article hold in more general spacetimes. The $(n+1)$-dimensional Schwarzschild spacetime with mass $m\geq 0$ is equipped with the metric \begin{equation}\label{Schwarzschild}
\bar{g} = - \lt( 1-\frac{2m}{r^{n-2}} \rt) dt^2 + \frac{1}{1-\frac{2m}{r^{n-2}}} dr^2 + r^2 g_{S^{n-1}}, \ r^{n-2}>2m.
\end{equation} It is the unique spherically symmetric spacetime that satisfies the vacuum Einstein equations. 
Let $Q=r dr \wedge dt$ be the conformal Killing-Yano two-form (see Definition \ref{C-K}) on the Schwarzschild spacetime. 
The curvature tensor of $\bar{g}$ can be expressed in terms of the conformal Killing-Yano two-form $Q$ (see Appendix C). We also denote the Levi-Civita connection of $\bar{g}$
by $D$.

Let $\Sigma$ be a closed oriented spacelike codimension-two submanifold and $\{e_a\}_{a=1,\cdots, n-1}$ be an oriented orthonormal frame of the tangent bundle. Let $\vec{H}$ denote the mean curvature vector of $\Sigma$. We assume the normal bundle of $\Sigma$ is also equipped with an orientation. Let $L$ be a null normal vector 
field along 
$\Sigma$. We define the connection one-form $\zeta_L$ with respect to $L$ by 
\begin{equation}\label{torsion_null0}\zeta_L (V)=\frac{1}{2}\langle D_V L, \underline{L} \rangle \quad \mbox{ for any tangent vector } V \in T\Sigma, \end{equation}
where $\underline{L}$ is another null normal such that $\langle L, \underline{L}\rangle=-2$. $\Sigma$ is said to be {\it torsion-free} with respect to $L$ if $\zeta_L=0$, or 
equivalently, $(D L)^\perp=0$ on $\Sigma$, where $(\cdot)^\perp$ denotes the normal component.

We prove the following Minkowski formula in the Schwarzschild spacetime.

\medskip
\begin{Theorem A} (Theorem \ref{first Minkowski identity})
Consider the two-form $Q = r dr \wedge dt$ on the Schwarzschild spacetime. For a closed oriented spacelike codimension-two submanifold $\Sigma$ in the Schwarzschild spacetime and a null normal vector field $\Lb$ along $\Sigma$, we have
\[ -(n-1) \int_\Sigma \langle \frac{\pl}{\pl t},\Lb \rangle \, d\mu+ \int_\Sigma Q(\vec{H},\Lb) \, d\mu +\sum_{a=1}^{n-1}\int_\Sigma Q(e_a, (D_{e_a} \Lb)^\perp) \,d\mu =0.
\] 
\end{Theorem A}

\medskip

If $\Sigma$ is torsion-free with respect to $\udl L$ for a null frame $L, \Lb$ that satisfies $\langle L, \underline{L}\rangle=-2$, the formula takes the form:
\begin{equation}\label{zero_order}-(n-1) \int_\Sigma \langle \frac{\pl}{\pl t},\Lb \rangle \, d\mu-\frac{1}{2} \int_\Sigma \langle \vec{H}, \udl{L}\rangle Q(L,\Lb) \, d\mu=0. 
\end{equation}
This formula corresponds to the $k=1$ case in \eqref{Minkowski.formula} (see \eqref{recover}) and  is proved in a more general setting, see Theorem \ref{first Minkowski identity}. The quantity $-\langle \vec{H}, L\rangle$  for a null normal $L$ corresponds to the null expansion of the surface in the direction of $L$. Codimension-two submanifolds play a special role in general relativity and their null expansions are closely related to gravitation energy as seen in Penrose's singularity theorem \cite{Penrose65}.

The Minkowski formula has been applied to various problems in global Riemannian geometry (see, for example, the survey paper \cite{Pigola-Rigoli-Setti03} and references therein). One important application is a proof of Alexandrov theorem which states that every closed embedded hypersurface of constant mean curvature (CMC) in $\R^n$ must be a round sphere. For the proof of Alexandrov Theorem and its generalization to various ambient manifolds using the Minkowski formula, see \cite{Brendle13, Montiel99, Montiel-Ros91, QX, Ros87}.
  
In general relativity, the causal future or past of a geometric object is of great importance. It is interesting to characterize when a surface lies in the null hypersurface generated by a ``round sphere." These are called ``shear-free" null hypersurfaces (see Definition \ref{shearfree}) in general relativity literature, and are analogues of umbilical hypersurfaces in Riemannian geometry.
\par


As an application of the Minkowski formula in the Schwarzschild spacetime, we give a characterization of spacelike codimension-two submanifolds in a null hypersurface of symmetry in terms of constant null expansion.
\medskip
\begin{Theorem B} \label{theoremB} (Theorem \ref{Alexandrov theorem for 2-surface in the Schwarzschild spacetime})
Let $\Sigma$ be a future incoming null embedded (see Definition \ref{future incoming null embedded}) closed spacelike codimension-two submanifold in the $(n+1)$-dimensional Schwarzschild spacetime. Suppose there is a future incoming null normal vector field $\underline{L}$ along $\Sigma$ such that $\langle \vec{H}, \underline{L}\rangle$ is a positive constant and $(D\underline{L})^\perp=0$. Then $\Sigma$ lies in a null hypersurface of symmetry.
\end{Theorem B}
\medskip

A natural substitute of CMC condition for higher codimensional submanifolds is to require the mean curvature vector field to be parallel as a section of the normal bundle. Yau \cite{Y} and Chen \cite{Chen} proved that a closed immersed spacelike 2-sphere with parallel mean curvature vector in the Minkowski spacetime must be a round sphere. We are able to generalize their results to the Schwarzschild spacetime.
\begin{Corollary C}(Corollary \ref{parallel})
Let $\Sigma$ be a closed embedded spacelike codimension-two submanifold with parallel mean curvature vector in the $(n+1)$-dimensional Schwarzschild spacetime. Suppose $\Sigma$ is both future and past incoming null embedded. Then $\Sigma$ is a sphere of symmetry. 
\end{Corollary C} 

Besides the Minkowski formula mentioned above, another important ingredient of the proof for Theorem B is a spacetime version of the Heintze-Karcher type inequality of Brendle \cite{Brendle13}. Brendle's inequality was used to prove the rigidity property of CMC hypersurfaces in warped product manifolds including the important case of a time slice in the Schwarzschild spacetime. Following his approach and generalizing a monotonicity
formula of his, we establish a spacetime version of this inequality (see Theorem \ref{Spacetime Heintze-Karcher inequality}) in Section 3.

\bigskip
Formula (\ref{zero_order}) can be viewed as a spacetime version of the Minkowski formula (\ref{Minkowski.formula}) with $k=1$.  In the second part of this paper, we take care of the case for general $k$. We introduce the notion of  mixed higher order mean curvature $P_{r, s}(\chi, \udl{\chi})$ for codimension-two submanifolds in spacetime, which generalizes
the notion of Weingarten curvatures for hypersurfaces in the Euclidean space. The mixed higher order mean curvature is derived from the two null second fundamental forms $\chi$ and $\udl{\chi}$ of the submanifold with respect to the null normals $L$ and $\Lb$, respectively. In the hypersurface case, there is only one second fundamental form $A$ and the Weingarten curvatures are defined as the elementary symmetric functions of $A$, $\sigma_k(A)$. Here, motivated by an idea of Chern \cite{Chern} (see also \cite{guan-shen}) in his study of Alexandrov's uniqueness theorem, we define the mixed higher order mean curvature $P_{r,s}(\chi, \udl{\chi})$, with $1\leq r+s \leq n-1$, see Definition \ref{prs}. It turns out that those quantities share some nice properties of $\sigma_k(A)$. First, we establish the following spacetime Minkowski formulae.

\begin{Theorem D}\label{minkowski0}(Theorem \ref{minkowski1})
Let $\Sigma$ be a closed spacelike codimension-two submanifold in a spacetime of constant curvature. Suppose $\Sigma$ is torsion-free with respect 
to the null frame $L$ and $\Lb$. Then
\begin{align}\label{Minkowski formula for higher order mean curvatureL} 
2\int_\Sigma P_{r-1, s}(\chi,\udl\chi)\langle  L, \frac{\partial}{\partial t}\rangle d\mu +\frac{r+s}{n-(r+s)} \int_\Sigma P_{r, s}(\chi,\udl\chi)Q(L, \udl{L}) d\mu =0 
\end{align}
and
\begin{align}\label{Minkowski formula for higher order mean curvatureLb} 
2\int_\Sigma P_{r, s-1}(\chi,\udl\chi)\langle  \Lb, \frac{\partial}{\partial t}\rangle d\mu -\frac{r+s}{n-(r+s)}  \int_\Sigma P_{r, s}(\chi,\udl\chi) Q(L, \udl{L}) d\mu =0. 
\end{align}
\end{Theorem D}
\eqref{zero_order} is a special case of \eqref{Minkowski formula for higher order mean curvatureLb} for torsion-free submanifolds ($r=0, s=1$) in a spacetime of constant curvature. Moreover, the classical Minkowski formulae (\ref{Minkowski.formula}) for hypersurfaces in Riemannian space forms (Euclidean space, hemisphere, hyperbolic space) can be recovered by \eqref{Minkowski formula for higher order mean curvatureL} and \eqref{Minkowski formula for higher order mean curvatureLb}, see \eqref{recover}.
As applications of these Minkowski formulae, we obtain Alexandrov type theorems with respect to mixed higher order mean curvature for torsion-free submanifolds in a spacetime of constant curvature, as a generalization of Theorem B.

 \begin{Theorem E} (Theorem \ref{Alexandrov theorem for Weingarten submanifolds in Minkowski spacetime})
Let $\Sigma$ be a past (future, respectively) incoming null embedded, closed spacelike codimension-two submanifold in an $(n+1)$-dimensional spacetime
of constant curvature. Suppose $\Sigma$ is torsion-free with respect to $L$ and $\Lb$ and the second fundamental form $\chi \in \Gamma_r$ ($-\chib \in \Gamma_s$, respectively). If $P_{r,0}(\chi, \udl\chi) = C$ ($P_{0,s} (\chi, \udl\chi)= (-1)^s C$, respectively) for some positive constant $C$ on $\Sigma$, then $\Sigma$ lies in a null hypersurface of symmetry.
\end{Theorem E}
Moreover, we show that a codimension-two submanifold in the Minkowski spacetime with $P_{r, s}(\chi, \udl\chi) = constant$ for $r>0, s>0$, satisfying other mild conditions, must be a sphere of symmetry. See Theorem \ref{Alexandrov for mixed higher mean curvature} for details.

\medskip
\par
In the proof of the spacetime Minkowski formulae (\ref{Minkowski formula for higher order mean curvatureL}) (\ref{Minkowski formula for higher order mean curvatureLb}), we make crucial use of a certain divergence property of $P_{r, s}(\chi, \udl \chi)$ for torsion-free submanifolds in spacetimes of constant curvature. Unfortunately, this property no longer holds for codimension-two submanifolds in the Schwarzschild spacetime because of non-trivial ambient curvature. However, under some assumption on the restriction of the conformal Killing-Yano two-form $Q$ to the submanifold, we establish integral inequalities (see Theorem \ref{Minkowski in Schwarzschild}) that imply Alexandrov type theorems in the Schwarzschild spacetime (see Corollary \ref{Alexandrov
in Schwarzschild}).

For 2-surfaces in the 4-dimensional Schwarzschild spacetime, we obtain a clean integral formula involving the total null expansion of $\Sigma$.

\par
\begin{Theorem F}(Theorem \ref{4dimensional Schwarzschild})
Consider the two-form $Q=rdr\wedge dt$ on the 4-dimensional Schwarzschild spacetime with $m\geq 0$. For a closed oriented spacelike 2-surface $\Sigma$, we have
\begin{align}
&2\int_\Sigma \langle \vec{H}, L\rangle \langle \udl L,\frac{\pl}{\pl t} \rangle d\mu \\\nonumber
= &- 16\pi m + \int_\Sigma \left\{\lt(R+ \frac{1}{4} \bar{R}_{L \Lb L \Lb} \rt) Q(L,\Lb) + \sum_{b,c=1}^2 \lt( \frac{1}{2} \bar{R}_{bc\Lb L} - 2(d\zeta_L)_{bc} \rt)Q_{bc} \right\}d\mu.
\end{align}
where $\zeta_L$ is the connection 1-form of the normal bundle with respect to $L$, $\bar{R}$ is the curvature tensor of the Schwarzschild spacetime, $R$ is the scalar curvature of $\Sigma$, and $Q_{bc}=Q(e_b, e_c)$, $(d\zeta_L)_{bc}=(d\zeta_L)(e_b, e_c)$, etc.
\end{Theorem F}

 The total null expansion $-\int_\Sigma \langle \vec{H}, L\rangle \langle \underline{L}, \frac{\partial}{\partial t}\rangle d\mu$ appears in the Gibbons-Penrose
inequality, see for example \cite{BW}.

\medskip
The rest of the paper is organized as follows. In section 2, we derive a simple case of spacetime Minkowski formula and give the proof of Theorem A. In section 3, we study a monotonicity formula and the spacetime Heintze-Karcher inequality. As an application, we prove the Alexandrov type theorems, Theorem B and Corollary C. In section 4, we introduce the notion of {\it mixed higher order mean curvatures} and establish spacetime Minkowski formulae for closed spacelike codimension-two submanifolds in constant curvature spacetimes, Theorem D. Moreover, we show the classical Minkowski formula (\ref{Minkowski.formula}) is recovered.
As an application of Theorem D, Alexandrov type theorems for submanifolds of constant mixed mean curvature in a spacetime of constant curvature are proved in section 5. In section 6, we generalize the integral formulae to the Schwarzschild spacetime. In particular, Theorem F is proved. At last,  the Appendix contains some computations used throughout the paper.

\medskip
{\bf Acknowledgements.} The authors would like to thank Simon Brendle, Po-Ning Chen, Shing-Tung Yau for helpful discussions. In particular, they learned a version of Theorem \ref{Spacetime Heintze-Karcher inequality} in the Minkowski spacetime from Brendle. The third-named author would also like to thank Pengfei Guan and Niky Kamran for their interest and penetrating remarks. Part of the work was done while all the authors were visiting the National Center for Theoretical Sciences and Taida Institute of Mathematical Science in Taipei, Taiwan. They would like to express their gratitude to the nice working environment that was provided during their visit. Finally, they wish to thank the referee for his/her valuable comments and suggestions.

\bigskip

\section{A Minkowski formula in spacetime}
\par
Let $F: \Sigma^{n-1} \rightarrow (V^{n+1},\langle,\rangle)$ be a closed immersed oriented spacelike codimension-two submanifold in an oriented (n+1)-dimensional Lorentzian manifold $(V^{n+1},\langle,\rangle).$ Denote the induced metric on $\Sigma$ by $\sigma$. We assume the normal bundle is also orientable and choose a coordinate system $\{u^a |\,\, a=1,2, \cdots,n-1\}$. We identify $\frac{\pl F}{\pl u^a}$ with $\frac{\pl}{\pl u^a}$, which is abbreviated as $\pl_a$. Let $D$ and $\na$ denote the Levi-Civita connection of $V$ and $\Sigma$ respectively. 

We recall the definition of conformal Killing-Yano two-forms.  
\begin{definition}\label{C-K}\cite[Definition 1]{JL}
Let $Q$ be a two-form on an $(n+1)$-dimensional pseudo-Riemannian manifold $(V, \langle,\rangle)$ with Levi-Civita connection $D$.  $Q$ is said to be a conformal Killing-Yano two-form if 
\begin{align}\label{CKY equation}
&(D_X Q)(Y,Z)+(D_Y Q)(X,Z) =\frac{2}{n}\left(\langle X, Y\rangle \langle \xi,Z \rangle - \frac{1}{2} \langle X, Z\rangle \langle \xi,Y \rangle - \frac{1}{2}\langle Y, Z\rangle \langle \xi,X \rangle \right)
\end{align}
for any tangent vectors $X$, $Y$ and $Z,$ where $\xi = div_V Q $.
\end{definition}
\medskip
In mathematical literature,  conformal Killing-Yano two-forms were introduced by Tachibana \cite{Tachibana69}, based on Yano's work on Killing forms. More generally, Kashiwada introduced the conformal Killing-Yano p-forms \cite{Kashiwada68}. 

It is well-known that there exists a conformal Killing-Yano two-form $Q$ on the Kerr spacetime with $\xi$ being a multiple of the stationary Killing vector $\frac{\pl}{\pl t}$, see \cite{JL}. We also show the existence of conformal Killing-Yano forms on a class of warped product manifolds in Appendix B.

\medskip

As mentioned in the introduction, we make use of the conformal Killing-Yano two-forms in Lorentzian manifolds to discover some new Minkowski formulae for the spacelike codimension-two submanifolds.

\begin{theorem}\label{first Minkowski identity}
Let $\Sigma$ be a closed immersed oriented spacelike codimension-two submanifold in an $(n+1)$-dimensional Riemannian or Lorentzian manifold $V$ that possesses a conformal Killing-Yano two-form $Q$. For any null normal vector field $\Lb$ of $\Sigma,$ we have
\begin{align}\label{k=1 Minkowski formula}
\frac{n-1}{n}\int_\Sigma \langle \xi,\Lb \rangle \, d\mu+ \int_\Sigma Q(\vec{H},\Lb) \, d\mu +\int_\Sigma Q(\pl_a, (D^a \Lb)^\perp) \,d\mu =0,
\end{align} where $\xi=div_V Q$.
\end{theorem}
\begin{proof}
Let $\underline{\chi}_{ab} = \langle D_a \Lb, \pl_b \rangle.$ Consider the one-form $\mathcal{Q} = Q(\pl_a,\Lb) du^a$ on $\Sigma$. We derive
\begin{align}\label{derivation}
\begin{split}
div_\Sigma \mathcal{Q} &= \na^a \mathcal{Q}_a - Q(\na^a \pl_a,\Lb) \\
&= (D^a Q)(\pl_a,\Lb) + Q(\vec{H},\Lb) + Q(\pl_a,D^a \Lb) \\
&= \frac{n-1}{n} \langle \xi,\Lb \rangle + Q(\vec{H},\Lb) + \underline{\chi}_{ab}Q^{ab} + Q(\pl_a,(D^a \Lb)^\perp) \\
&= \frac{n-1}{n} \langle \xi,\Lb \rangle + Q(\vec{H},\Lb) +  Q(\pl_a,(D^a \Lb)^\perp).
\end{split}
\end{align}
The assertion follows by integrating over $\Sigma.$ 
\end{proof}
In the case of the Schwarzschild spacetime, we take $Q=rdr\wedge dt$, then $\xi= -n\frac{\partial}{\partial t}$ and Theorem A follows from the general formula (\ref{k=1 Minkowski formula}).

\section{ An Alexandrov Theorem in spacetime}
\medskip

\subsection{A monotonicity formula}

\mbox{}\\

In this section, we assume that $\Sigma$ is a spacelike codimension-two submanifold with spacelike mean curvature vector in a Lorentzian manifold $V$ that possesses a conformal Killing-Yano two-form $Q$. We fix the sign of $Q$ by requiring $\xi := div_V Q$ to be past-directed timelike. Let $\Lb$ be a future incoming null normal and $L$ be the null normal with $\langle L,\Lb \rangle = -2$.   We note that the choice of $L,\Lb$ is unique up to the scaling $L \rw aL, \Lb \rw \frac{1}{a} \Lb$ by a function $a$ on $\Sigma$. Define the null second fundamental forms with respect to $L, \Lb$ by
\begin{align*}
\chi_{ab} &=\langle D_{\partial_a} L, \partial_b\rangle, \ \ \ \  \underline{\chi}_{ab} =\langle D_{\partial_a}\udl{ L}, \partial_b\rangle.
\end{align*}
Suppose $\langle \vec{H}, \Lb\rangle\not=0$ on $\Sigma$, define the functional
\begin{equation}\label{F_functional}
\mathcal{F}(\Sigma, [\Lb]) = \frac{n-1}{n} \int_\Sigma \dfrac{\langle \xi,\Lb \rangle}{\langle \vec{H},\Lb \rangle} d\mu - \frac{1}{2}\int_\Sigma Q(L,\Lb) d\mu.
\end{equation}
Note that $\mathcal{F}$ is well-defined in that it is invariant under the change $L \rw aL, \Lb \rw \frac{1}{a} \Lb.$ 

Let $\underline{C}_0$ denote the future incoming null hypersurface of $\Sigma$. $C_0$ is obtained by taking the collection
of all null geodesics emanating from $\Sigma$ with initial velocity $\underline{L}$.
We also extend $\Lb$ arbitrarily to a future-directed null vector field along $\underline{C}_0,$ still denoted by $\Lb$. Consider the evolution of $\Sigma$ along $\underline{C}_0$ by a family of immersions $F: \Sigma \times [0,T) \rw \underline{C}_0$ satisfying
\begin{align}\label{flow}
\left\{ \begin{array}{rl}
\frac{\pl F}{\pl s}(x, s) \! \!\!\!&= \vphi(x,s)\Lb \\
F(x,0) \!\!\!\!&= F_0(x).
\end{array} \right.
\end{align}
for some positive function $\vphi(x,s).$

We recall the following spacetime curvature condition \cite[page 95]{Hawking-Ellis} which plays an important role in the 
monotonicity formula. 

\begin{definition}
A Lorentzian manifold is said to satisfy the null convergence condition if
\begin{equation}\label{null_convergence}
Ric(L,L) \geq 0 \quad \mbox{for any null vector} \, L.
\end{equation}
\end{definition}
This is exactly the curvature assumption made in Penrose's  celebrated singularity theorem \cite[page 263, Theorem 1]{Hawking-Ellis}.

We prove the following monotonicity formula along the flow $F$:
\begin{theorem}\label{monotonicity}
Let $F_0:\Sigma \rw V$ be an immersed closed oriented spacelike codimension-two submanifold in a Lorentzian manifold $V$ with a conformal 
Killing-Yano two-form $Q$ that satisfies either one of the following assumptions
\begin{enumerate}\label{assumption: null convergence condition}
\item $V$ is vacuum (possibly with cosmological constant). 
\item $\xi=div_V Q$ is a Killing field and $V$ satisfies the null convergence condition.
\end{enumerate} Suppose that $\langle \vec{H},\Lb \rangle >0$ on $\Sigma$ for some future-directed incoming null normal vector field $\Lb.$  
Then $\mathcal{F}(F(\Sigma,s), [\udl L])$ is monotone decreasing along the flow. 
\end{theorem}
\begin{proof}
We start with the evolution of the first term of $\mathcal{F}$. Suppose $D_{\Lb} \Lb = \omegab \Lb$ for a function $\omegab$. The Raychaudhuri equation \cite[(9.2.32)]{Wald} implies
\begin{equation}\label{positive_null_expansion}
\begin{split}
\frac{\pl}{\pl s} \langle \vec{H},\Lb \rangle &= \vphi \left( |\chib|^2 	+\omegab \langle \vec{H}, \Lb \rangle + Ric(\Lb,\Lb) \right)\\
&\geq \vphi \left( |\chib|^2 +\omegab \langle \vec{H}, \Lb \rangle \right),
\end{split}
\end{equation}
where $|\chib|^2 = \chib^{ab} \chib_{ab}$.
On the other hand,
\begin{align*}
\frac{\pl}{\pl s} \langle \xi,\Lb \rangle = \vphi \left( \langle D_{\Lb} \xi,\Lb \rangle + \omegab \langle \xi,\Lb \rangle \right).
\end{align*} 
If $V$ satisfies assumption (1), by \cite[equation (19)]{JL}, we have
\begin{align*}
\langle D_{\Lb} \xi,\Lb \rangle = \frac{n}{n-1} Ric_{a\Lb} \tensor{Q}{^a_\Lb} =0.
\end{align*}
If $V$ satisfies assumption (2), $\langle D_{\Lb} \xi,\Lb \rangle$ also vanishes since $\xi$ is a Killing vector. Recall that $\xi$ is past-directed timelike. Hence $\langle \xi, \Lb \rangle > 0$ and by the Cauchy-Schwartz inequality, 
\begin{align}\label{evolution of f over mean curvature}
\frac{\pl}{\pl s} \int_\Sigma \dfrac{\langle \xi,\Lb \rangle}{\langle \vec{H},\Lb \rangle} d\mu &\le -\int_\Sigma \vphi \left[ \dfrac{\langle \xi,\Lb \rangle}{(\mxtr\chib)^2} |\chib|^2 + \langle \xi,\Lb \rangle \right] d\mu \leq -\frac{n}{n-1} \int_\Sigma \vphi \langle \xi,\Lb \rangle d\mu.
\end{align}

The evolution of the second term of $\mathcal{F}$ is given by
\begin{align*}
&\frac{\pl}{\pl s} \int_\Sigma Q(L,\Lb) d\mu\\
 =& \int_\Sigma \left[ \vphi \left( D_{\Lb} Q \right)(L,\Lb) + Q(D_{\pl_s} L,\Lb) + Q(L,D_{\pl_s} \Lb)  - \vphi Q(L,\Lb)\langle \vec{H},\Lb \rangle \right] d\mu. \notag
\end{align*}
From the conformal Killing-Yano equation (\ref{CKY equation}), we derive
\begin{align*}
\left( D_{\Lb} Q \right)(L,\Lb) = \frac{1}{n} \langle \xi,\Lb \rangle \langle L,\Lb \rangle = -\frac{2}{n} \langle \xi,\Lb \rangle.
\end{align*}
On the other hand, by standard computation
\begin{align*}
\langle D_{\pl_s} L,\Lb \rangle &= -\langle L, \vphi \underline{\omega} \Lb \rangle, \\
\langle D_{\pl_s} L ,\pl_a \rangle &= - \langle L, D_a (\vphi \Lb) \rangle= 2 \na_a \vphi - \vphi \langle L,D_a \Lb \rangle, 
\end{align*}
we have
\[D_{\pl_s} L = \lt( 2 \na^a \vphi - \vphi \langle L, D^a \Lb \rangle \rt) \pl_a -\vphi \omegab L. \]
The computations together yield
\begin{align*}
&Q(D_{\pl_s} L,\Lb) + Q(L,D_{\pl_s} \Lb)  - \vphi Q(L,\Lb)\langle \vec{H},\Lb \rangle\\  
 &= 2 \na^a \vphi Q(\pl_a,\Lb) + 2\vphi Q(\pl_a, (D_a \Lb)^\perp) + 2\vphi Q(\vec{H},\Lb) \\
&= 2 \na^a \left( \vphi Q(\pl_a,\Lb) \right) - \frac{2(n-1)}{n} \vphi \langle \xi,\Lb \rangle.
\end{align*}
In the last equality, we make use of (\ref{derivation}). Consequently, we obtain
\begin{align}
\frac{\pl}{\pl s} \int_\Sigma Q(L,\Lb) d\mu = -2 \int_\Sigma \vphi\langle \xi,\Lb \rangle d\mu. \label{evolution of X dot nu}
\end{align}
Then, the assertion follows from (\ref{evolution of f over mean curvature}) and (\ref{evolution of X dot nu}).
\end{proof}

\medskip
\subsection{A spacetime CMC condition}

\mbox{}\\

Hypersurfaces of constant mean curvature (CMC) provide models for soap bubbles, and have been studied extensively for a long time. A common generalization of this condition for higher codimensional submanifolds is the parallel mean curvature condition.  In general relativity, the most relevant physical phenomenon is the divergence of light rays emanating from a codimension-two submanifold. This is called the null expansion in physics literature. We thus impose constancy conditions on the null expansion of codimension-two submanifolds. More precisely, we are interested in the codimension-two submanifold that admits a future null normal vector field $L$ such that
\begin{enumerate}
\item $\langle \vec{H},L \rangle$ is a constant, and
\item $(DL)^\perp=0$ on $\Sigma$. 
\end{enumerate}

We review the definition of connection one-form of mean curvature gauge from \cite{Wang-Yau08}  and relate it to the condition introduced above. 
\begin{definition}\label{connection one-form}
 Let $\vec{H}$ denote the mean curvature vector of $\Sigma$. Let $\{e_n, e_{n+1}\}$ be an oriented orthonormal
 frame of the normal bundle such that $e_{n}$ is spacelike and $e_{n+1}$ is future timelike.  We define the normal vector field 
 $\vec{J}$ by reflecting $\vec{H}$ along the incoming light cone:
\begin{align*}
\vec{J} = \langle \vec{H},e_{n+1}\rangle e_n - \langle \vec{H},e_n\rangle e_{n+1} 
\end{align*} as in \cite{Wang-Yau08}.
The connection one-form $\zeta_{e_n}$ of the normal bundle with respect to $e_n$ and $e_{n+1}$ is defined by
\begin{equation}\label{torsion_orthonormal}
\zeta_{e_n}(V) = \langle D_V e_n, e_{n+1} \rangle \quad \mbox{ for any tangent vector } V \in T\Sigma.
\end{equation} Suppose the mean curvature vector is spacelike, we take $e_n^{\vec{H}} = - \frac{\vec{H}}{|\vec{H}|}$ and $e_{n+1}^{\vec{H}} = \frac{\vec{J}}{|\vec{H}|}$ and write $\alpha_{\vec{H}}$ for the connection one-form with respect to this mean curvature gauge: 
\[ \alpha_{\vec{H}}(V)=\langle D_V e_n^{\vec{H}}, e_{n+1}^{\vec{H}}\rangle. \] This is consistent with the notation in 
\cite{CWW}. We note that $\{e_n^{\vec{H}}, e_{n+1}^{\vec{H}}\}$ determines the same orientation as $\{e_n, e_{n+1}\}$.
\end{definition}
Recall the connection one-form with respect to  $L$ is given by 
\begin{equation}\label{torsion_null}\zeta_L(V)=\frac{1}{2}\langle D_V L, \underline{L} \rangle \quad \mbox{ for any tangent vector } V \in T\Sigma.\end{equation}

The two definitions in \eqref{torsion_orthonormal} and \eqref{torsion_null} give the same connection one-form if we choose $L=e_{n+1}+e_n$ and $\underline{L}=e_{n+1}-e_n$. 

\begin{proposition} \label{single equation}
Suppose the mean curvature vector field $\vec{H}$ of $\Sigma$ is spacelike.
\begin{enumerate}
\item If $\langle \vec{H},L \rangle =-c <0$ and $(DL)^\perp =0$ for some future outward null normal $L$ and some positive constant $c$, then $\alpha_{\vec{H}} = -d \log|\vec{H}|$.
\item If $\langle \vec{H},\Lb \rangle =c >0 $ and $(D\Lb)^\perp=0$ for some future inward null normal $\Lb$ and some positive constant $c$, then $\alpha_{\vec{H}} = d \log|\vec{H}|$.
\end{enumerate}
\end{proposition}

\begin{proof}
Recall that the dual mean curvature vector $\vec{J}$ is future timelike. For (1), the condition $\langle \vec{H},L \rangle = -c <0$ is equivalent to 
\[ L = \frac{c}{|\vec{H}|} \left( -\frac{\vec{H}}{|\vec{H}|} + \frac{\vec{J}}{|\vec{H}|} \right). \]
Choose $\Lb = \frac{|\vec{H}|}{c} \left( \frac{\vec{H}}{|\vec{H}|} + \frac{\vec{J}}{|\vec{H}|} \right)$ such that $\langle L, \Lb\rangle=-2$.
Since $(D L)^\perp=0$, we have
\begin{align*}
0 &=\frac{1}{2}\langle D_a L,\Lb \rangle = \frac{1}{2} \pl_a \left( \frac{c}{|\vec{H}|} \right) \frac{|\vec{H}|}{c} (-2) + \left\langle D_a \left( -\frac{\vec{H}}{|\vec{H}|} \right), \frac{\vec{J}}{|\vec{H}|} \right\rangle \\
&= \pl_a \log |\vec{H}| + \left\langle D_a \left( - \frac{\vec{H}}{|\vec{H}|} \right), \frac{\vec{J}}{|\vec{H}|} \right\rangle.
\end{align*}
Hence $\alpha_{\vec{H}} = - d\log |\vec{H}|$. (2) can be proved in a similar way.
\end{proof}

\begin{remark}
We remark that when $\Sigma$ lies in a totally geodesic time slice of a static spacetime (see \cite[page 119]{Wald} for the definition),  the condition (1) and (2) reduces to the CMC condition.
\end{remark}

\medskip
\subsection{A Heintze-Karcher type inequality} $\\$
\par
In this and the next subsections, we focus our discussion on a class of spherically symmetric static spacetimes which includes the Schwarzschild spacetime. 
\begin{assumption}\label{spacetime_assumption} We assume $V$ is a spacetime that satisfies the null convergence condition \eqref{null_convergence} and the metric $\bar{g}$ on $V = \mathbb{R} \times M$ is of the form  
\begin{align}
\bar{g} = -f^2(r) dt^2 + \frac{1}{f^2(r)} dr^2 + r^2 g_{N}.
\end{align}
where $M=[r_1, r_2)\times N$ equipped with metric 
\begin{eqnarray}\label{metricg}
g=\frac{1}{f^2(r)} dr^2 + r^2 g_{N}
\end{eqnarray}
 and $(N, g_N)$ is a compact $n$-dimensional Riemannian manifold. We consider two cases: 
\begin{enumerate}
\item[(i)] $f:[0,\infty)\rw\mathbb{R}$ with $f(0)=1$, $f'(0)=0$, and $f({r})>0$ for $r\geq 0$.
\item[(ii)] $f: [r_0,\infty)\rw \mathbb{R}$ with $f(r_0)=0$ and $f(r)>0$ for $r>r_0.$
\end{enumerate}
\end{assumption}

In case (i), $(V,\bar{g})$ is complete. In case (ii), $V$ contains an event horizon $\mathcal{H} = \{r=r_0\}$. 
We note that the warped product manifolds considered in \cite{Brendle13} are embedded as totally geodesic slices in these spacetimes. 

\begin{remark} \label{remark_assumption}For a spacetime $V$ that satisfies Assumption \ref{spacetime_assumption}, a simple calculation shows that $Q=rdr\wedge dt$ is a conformal Killing-Yano two-form and $div_V Q=\xi= -n\frac{\partial}{\partial t}$
is a Killing field. 
\end{remark}

\begin{lemma}\label{null convergence condition implies H3'}
Let $(M, g)$ be a time slice in $V$. The null convergence condition of $(V,\bar{g})$ is equivalent to 
\begin{align}\label{static inequality}
(\Delta_g f) g - \mbox{Hess}_g f + f \mbox{Ric}(g) \geq 0 
\end{align}
on $M$.
\end{lemma}
\begin{proof} 
O'Neill's formula in our case reduces to (see \cite[Proposition 2.7]{Co})
\begin{align*}
\mbox{Ric}(\bar{g})(v,w) &= \mbox{Ric}(g)(v,w) - \frac{\mbox{Hess}_g f(v,w)}{f},\\
\mbox{Ric}(\bar{g}) \lt( v, \frac{\pl}{\pl t} \rt) &=0, \\
\mbox{Ric}(\bar{g}) \lt( \frac{\pl}{\pl t}, \frac{\pl}{\pl t} \rt) &= - \frac{\Delta_g f}{f} \bar{g}(\frac{\pl}{\pl t}, \frac{\pl}{\pl t}),
\end{align*}
for any tangent vectors $v$ and $w$ on $M$. 
\par
Note that $\bar{g} \lt( \frac{\pl}{\pl t}, \frac{\pl}{\pl t}\rt) = -f^2$. Hence, a unit tangent vector $v$ on $M$ gives rise to a null vector $L= \frac{1}{f} \frac{\pl}{\pl t} + v$ in the spacetime and any null vector along $M$ is a multiple of $\frac{1}{f} \frac{\pl}{\pl t} + v $ for some unit tangent vector $v$ on $M$. As a result, 
\begin{align*}
\mbox{Ric}(\bar{g})(L,L) \quad and \quad \frac{\Delta_g f}{f} + \mbox{Ric}(g)(v,v) -\frac{\mbox{Hess}_g f(v,v)}{f} 
\end{align*}
  have the same sign. Therefore, the null convergence condition is equivalent to (\ref{static inequality}). 
\end{proof}

\medskip
 In \cite{Brendle13}, Brendle proved a Heintze-Karcher-type inequality for mean convex hypersurfaces in $(M, g)$. In our context, it is as in the following theorem.
\begin{theorem}\cite{Brendle13}
Let $S$ be a smooth, closed, embedded, orientable hypersurface in a time slice of a spacetime $V$ that satisfies Assumption \ref{spacetime_assumption}. Suppose that $S$ has positive mean curvature $H>0$ in the slice. Then
\begin{align}\label{Brendle's Heintze-Karcher inequality}
(n-1) \int_S \frac{f}{H} d\mu \geq \int_S \langle X,\nu \rangle d\mu,
\end{align} where $\nu$ is the outward unit normal of $S$ in the slice and $X=r f \frac{\pl}{\pl r}$ is the conformal Killing vector field on the slice.
Moreover, equality holds if and only if $S$ is umbilical.
\end{theorem}
\begin{proof}
We first remark that since $S$ is embedded and orientable, $S$ is either null-homologous or homologous to $\{r_0\} \times N$. Hence $\pl \Omega = S$ or $\pl \Omega=S - \{r_0\} \times N$ for some domain $\Omega \subset M$. Inequality \eqref{Brendle's Heintze-Karcher inequality} is equivalent to the one in Theorem 3.5 and the one in  Theorem 3.11 of Brendle's paper in the respective cases. For the reader's convenience, we trace Brendle's argument leading to (\ref{Brendle's Heintze-Karcher inequality}). 

The assumptions on $(M,g)$ are listed in page 248 \cite{Brendle13}:
\begin{align}\label{assumption on RicN}
\mbox{Ric}_N \geq (n-2)\rho g_N
\end{align} 
and (H1)-(H3) (note that condition (H4) is not used in the proof of (\ref{Brendle's Heintze-Karcher inequality})). While Brendle writes the metric in geodesic coordinates
\begin{align*}
d\bar{r} \otimes d\bar{r} + h^2(\bar{r})g_N,
\end{align*}
it is equivalent to (\ref{metricg}) by a change of variables $r=h$ and $f=\frac{dh}{d\bar{r}}$. Moreover, as explained in the beginning of section 2 (page 252 of \cite{Brendle13}), (H1) and (H2) are equivalent to our assumptions (i) and (ii) on $f$. In Proposition 2.1, (\ref{assumption on RicN}) and (H3) together imply that (\ref{static inequality}) holds on $(M,g)$. 

The condition \eqref{static inequality} turns out to be the only curvature assumption that is necessary in proving (\ref{Brendle's Heintze-Karcher inequality}). More precisely, (\ref{static inequality}) is used to prove the key monotonicity formula, Proposition 3.2 (page 256). Inequality (\ref{Brendle's Heintze-Karcher inequality}) is a direct consequence of Proposition 3.2 up to several technical lemmata, Lemma 3.6 to Corollary 3.10, in which only assumptions (H1) and (H2) are used.

Finally, the inequalities appeared in Theorem 3.5 and Theorem 3.11 in \cite{Brendle13} are equivalent to (\ref{Brendle's Heintze-Karcher inequality}) by the divergence theorem.   
\end{proof}

\medskip
Before stating the spacetime Heintze-Karcher inequality, we define the notions of future incoming null embeddedness and shear-free null hypersurface.
\begin{definition}\label{future incoming null embedded}
A closed embedded spacelike codimension-two submanifold $\Sigma$ in a static spacetime $V$ is {\it future (past, respectively) incoming null embedded} if the future (past, respectively) incoming null hypersurface of $\Sigma$ intersects a totally geodesic time-slice $\mathcal{M}_T=\{ t=T\}\subset V$ at a smooth, embedded, orientable hypersurface $S$. 
\end{definition}

\begin{definition}\label{shearfree}
An incoming null hypersurface $\mathcal{\underline{C}}$ is {\it shear-free} if there exists a spacelike hypersurface $\Sigma$ in  $\mathcal{\underline{C}}$ such that the null second fundamental form $\chib_{ab} = \langle D_a \Lb, \pl_b \rangle$ of $\Sigma$ with respect to some null normal $\Lb$ satisfies $\chib_{ab} = \psi \sigma_{ab}$ for some function $\psi$. A shear-free outgoing null hypersurface is defined in the same way.
\end{definition}
Note that being shear-free is a property of the null hypersurface, see \cite[page 47-48]{Sauter08}.
The spacetime Heintze-Karcher inequality we prove is the following: 
\begin{theorem}\label{Spacetime Heintze-Karcher inequality}
Let $V$ be a spacetime as in Assumption \ref{spacetime_assumption}. Let $\Sigma \subset V$ be a future incoming null embedded closed spacelike codimension-two submanifold with  $\langle \vec{H},\Lb \rangle >0$ where $\Lb$ is a future incoming null normal.   Then
\begin{align}\label{Heintze-Karcher inequality}
-(n-1) \int_\Sigma \dfrac{\langle \frac{\pl}{\pl t},\Lb \rangle}{\langle \vec{H},\Lb \rangle} d\mu - \frac{1}{2} \int_\Sigma Q(L,\Lb) d\mu \geq 0,
\end{align} for a future outgoing null normal $L$ with $\langle L, \Lb\rangle=-2$  and $Q=rdr\wedge dt$ is the conformal Killing-Yano two-form on $V$. Moreover, the equality holds if and only if $\Sigma$ lies in an incoming shear-free null hypersurface.
\end{theorem}
\begin{proof}
We arrange $\vphi$ in (\ref{flow}) such that $\underline{\omega} \geq 0$ and that $F(\Sigma,1)=S$, the smooth hypersurface defined in Definition \ref{future incoming null embedded}. We first claim that $S \subset \mathcal{M}_T$ has positive mean curvature, $H>0$. Recall that Raychaudhuri equation implies 
\begin{equation}
\begin{split}
\frac{\pl}{\pl s} \langle \vec{H},\Lb \rangle = \vphi \left( |\chib|^2 	+\omegab \langle \vec{H}, \Lb \rangle + Ric(\Lb,\Lb) \right)\geq \vphi \left( |\chib|^2 +\omegab \langle \vec{H}, \Lb \rangle \right),
\end{split}
\end{equation}
and hence $\langle \vec{H}, \Lb \rangle >0$ on $S$. We choose $\Lb = \frac{1}{f} \frac{\pl}{\pl t} - e_n$ on $S$, where $e_n$ is the outward unit normal of $S$ with respect to $\Omega$, and compute 
\begin{align*}
\left\langle \vec{H}, \frac{1}{f}\frac{\pl}{\pl t} - e_n \right\rangle = H.
\end{align*}
\par

The claim follows since the positivity of $\langle \vec{H}, \Lb \rangle$ is independent of the scaling of $\Lb$.
Next we choose $L = \frac{1}{f} \frac{\pl}{\pl t} + e_n$ on $S$ and compute
\begin{align*}
-\left\langle \frac{\pl}{\pl t}, \frac{1}{f}\frac{\pl}{\pl t} - e_n \right\rangle &= f, \\
Q \left(\frac{1}{f}\frac{\pl}{\pl t} + e_n, \frac{1}{f}\frac{\pl}{\pl t} - e_n \right) &= 2\langle X,e_n \rangle.
\end{align*}

Remark \ref{remark_assumption} implies the monotonicity formula (Theorem \ref{monotonicity}) holds with $\xi=-n\frac{\partial}{\partial t}$
and thus  
\begin{align}\label{spacetime Heintze-Karcher inequality: after monotonicity}
&-(n-1) \int_\Sigma \dfrac{\langle \frac{\pl}{\pl t},\Lb \rangle}{\langle \vec{H},\Lb \rangle} d\mu - \frac{1}{2} \int_\Sigma Q(L,\Lb) d\mu\\
&\geq  -(n-1) \int_{F(\Sigma,1)} \dfrac{\langle \frac{\pl}{\pl t},\Lb \rangle}{\langle \vec{H},\Lb \rangle} d\mu - \frac{1}{2} \int_{F(\Sigma,1)} Q(L,\Lb) d\mu. \notag\end{align}

As $F(\Sigma, 1)=S$, the above calculation on $S$ shows the last expression is equal to 
\[ (n-1) \int_S \frac{f}{H} d\mu - \int_S \langle X,e_n \rangle d\mu\geq 0 \] by  \eqref{Brendle's Heintze-Karcher inequality}. Moreover, $S$ is umbilical if the equality holds. Hence the future incoming null hypersurface generated from $\Sigma$ is shear-free.  
\end{proof}

\medskip
By reversing the time orientation, we also obtain the Heintze-Karcher inequality for past incoming null smooth submanifolds.
\begin{theorem}\label{Spacetime Heintze-Karcher inequality: past incoming}
Let $V$ be a spacetime as in Assumption \ref{spacetime_assumption}. Let $\Sigma \subset V$ be a past incoming null embedded, closed spacelike codimension-two submanifold such that $\langle \vec{H},  L\rangle < 0$ with respect to some future outgoing null normal $ L$.   Then
\begin{align}\label{Heintze-Karcher inequality1}
(n-1) \int_\Sigma \dfrac{\langle \frac{\partial}{\partial t},  L \rangle}{\langle \vec{H}, L\rangle} d\mu - \frac{1}{2} \int_\Sigma Q(L,\Lb) d\mu \geq 0,
\end{align} for a future incoming null normal $\udl L$ with $\langle L, \Lb\rangle=-2$ and $Q=rdr\wedge dt$ is the conformal Killing-Yano two-form on $V$.
Moreover, the equality holds if and only if $\Sigma$ lies in an outgoing shear-free null hypersurface.
\end{theorem}
\begin{proof}
When we reverse the time orientation, we replace $\frac{\pl}{\pl t}$ by $-\frac{\pl}{\pl t}$, $Q$ by $-Q$, $\Lb$ by $-L$, and $L$ by $-\Lb$. Plug these into (\ref{Heintze-Karcher inequality}) and we obtain (\ref{Heintze-Karcher inequality1}).
\end{proof}
\medskip
\subsection{A Spacetime Alexandrov Theorem}

\mbox{}\\

Together with the spacetime Heintze-Karcher inequality proved in the previous subsection, the Minkowski formula (\ref{k=1 Minkowski formula}) implies the following spacetime Alexandrov type theorem.

\begin{theorem}
Let $V$ be a spherically symmetric spacetime as in Assumption \ref{spacetime_assumption} and $\Sigma$ be a future incoming null embedded, closed embedded, spacelike codimension-two submanifold in $V$. Suppose there is a future incoming null normal vector field $\underline{L}$ along $\Sigma$ such that $\langle \vec{H}, \underline{L}\rangle$ is a positive constant and $(D\underline{L})^\perp=0$ on $\Sigma$.Then $\Sigma$ lies in a
shear-free null hypersurface.\end{theorem}

\begin{proof}
Write $\vec{H} = -\frac{1}{2} \langle \vec{H},\Lb \rangle L - \frac{1}{2} \langle \vec{H}, L \rangle \Lb$. From the assumption, $(D_a \Lb)^\perp=0$, the spacetime Minkowski formula (\ref{k=1 Minkowski formula}) becomes
\begin{align*}
-(n-1) \int_\Sigma \langle \frac{\pl}{\pl t},\Lb \rangle d\mu - \frac{1}{2} \int_\Sigma \langle \vec{H},\Lb \rangle Q(L,\Lb)=0.
\end{align*}
\par

Again from the assumption, $\langle \vec{H},\Lb \rangle$ is a positive constant function and we can divide both sides by $\langle \vec{H},\Lb \rangle$ to get
\begin{align*}
-(n-1) \int_\Sigma \dfrac{\langle \frac{\pl}{\pl t}, \Lb\rangle}{\langle \vec{H},\Lb \rangle} d\mu - \frac{1}{2} \int_\Sigma Q(L,\Lb) \rangle d\mu =0.
\end{align*}
Hence the equality is achieved in the spacetime Heintze-Karcher inequality (\ref{Heintze-Karcher inequality}) and we conclude that $\Sigma$ lies in a shear-free null hypersurface. 
\end{proof}

\medskip

 \begin{definition}\label{null hypersurface of symmetry} A null hypersurface in an $(n+1)$-dimensional spherically symmetric spacetime is called a {\it null hypersurface of symmetry} if it is invariant under the $SO(n)$ isometry that defines the spherical symmetry. In other words, it is generated by a sphere of symmetry.  
\end{definition}
We remark that in the Lorentzian space forms, the additional boost isometry sends a null hypersurface of symmetry into another one defined by a conjugate $SO(n)$-action.

An important example of the spacetime satisfying  Assumption \ref{spacetime_assumption} is the exterior Schwarzschild spacetime for which the metric has the form \eqref{Schwarzschild}. Since the spheres of symmetry are the only closed umbilical hypersurfaces in the totally geodesic time slice of the Schwarzschild spacetimes \cite[Corollary 1.2]{Brendle13}, as a direct corollary of the above spacetime Alexandrov theorem, we obtain   

\begin{theorem}[Theorem B]\label{Alexandrov theorem for 2-surface in the Schwarzschild spacetime}
Let $\Sigma$ be a future incoming null embedded (see Definition \ref{future incoming null embedded}) closed embedded spacelike codimension-two submanifold in the $(n+1)$-dimensional Schwarzschild spacetime. Suppose there is a future incoming null normal vector field $\underline{L}$ along $\Sigma$ such that $\langle \vec{H}, \underline{L}\rangle$ is a non-zero constant and $(D\underline{L})^\perp=0$. Then $\Sigma$ lies in a null hypersurface of symmetry.
\end{theorem}

\medskip
As observed in Proposition \ref{single equation}, the condition in the above theorem can be characterized in terms of the norm of the mean curvature vector and the connection one-form in the mean curvature gauge. 

\begin{Theorem B'}
Let $\Sigma$ be a future incoming null embedded (see Definition \ref{future incoming null embedded}) closed embedded spacelike codimension-two submanifold in the Schwarzschild spacetime with spacelike mean curvature $\vec{H}$. Suppose $\alpha_{\vec{H}} = d \log|\vec{H}|$ on $\Sigma$. Then $\Sigma$ lies in a null hypersurface of symmetry.
\end{Theorem B'}

Finally, we generalize a result of  Yau \cite{Y} and Chen \cite{Chen} to the Schwarzschild spacetime.
\begin{corollary}[Corollary C]\label{parallel}
Let $\Sigma$ be a closed embedded spacelike codimension-two submanifold with parallel mean curvature vector in the Schwarzschild spacetime. Suppose $\Sigma$ is both future and past incoming null embedded. Then $\Sigma$ is a sphere of symmetry.
\end{corollary}
\begin{proof}
The condition of parallel mean curvature vector implies $|\vec{H}|$ is constant and $\alpha_{\vec{H}}$ vanishes. The previous theorem implies $\Sigma$ is the intersection of one incoming and one outgoing null hypersurface of symmetry. Therefore, $\Sigma$ is a sphere of symmetry.
\end{proof}

\bigskip
\section{General Minkowski Formulae for mixed Higher order Mean Curvature}
\par

In this section, we introduce the notion of mixed higher order mean curvature of a codimension-two submanifold $\Sigma$ in a spacetime $V$ of dimension
$(n+1)$. 
Let $\bar{R}$ denote the curvature tensor of $V$. Let $L$ and $\underline{L}$ be two null normals of $\Sigma$ such that $\langle L, \underline{L}\rangle=-2$.  Recall the null second fundamental forms with respect to $L, \Lb$:
\begin{align*}
\chi_{ab}=\langle D_{\partial_a} L, \partial_b\rangle \\
\underline{\chi}_{ab}=\langle D_{\partial_a}\udl{ L}, \partial_b\rangle
\end{align*}
and write $\zeta=\zeta_L$ for the connection 1-form with respect to $L$:
\[ \zeta_a = \frac{1}{2} \langle D_{\partial_a} L,\Lb \rangle. \] 

\medskip
\begin{definition} \label{prs}
For any two non-negative integers $r$ and $s$ with $0\le r+s\leq n-1$,  the \text{mixed higher order mean curvature} $P_{r,s} (\chi, \udl{\chi})$ with respect to $L$ and $\Lb$ is defined through the following expansion: 
\begin{equation} \label{det}
\det(\sigma+y \chi+\udl{y}\udl{\chi})=\sum_{0\leq r+s\leq n-1}\frac{(r+s)!}{r! s!}y^r \udl{y}^s P_{r,s}(\chi, \udl{\chi}), \end{equation}
where $y$ and $\udl{y}$ are two real variables and $\sigma$ is the induced metric on $\Sigma$. 
We also define symmetric 2- tensors $T_{r,s}^{ab} (\chi, \udl{\chi})$ and $\udl{T}_{r,s}^{ab} (\chi, \udl{\chi}) $ on $\Sigma$ by
\[T_{r,s}^{ab}(\chi, \udl{\chi})=\frac{\delta P_{r,s} (\chi, \udl{\chi})}{\delta \chi_{ab}}\text{   and   }\ \udl{T}_{r,s}^{ab} (\chi, \udl{\chi})=\frac{\delta P_{r,s} (\chi, \udl{\chi})}{\delta \udl{\chi}_{ab}}.\]

\end{definition}
In the following, we write $P_{r,s}$ for $P_{r,s}(\chi,\chib)$, $T_{r,s}^{ab}$ for $T_{r,s}^{ab}(\chi, \udl{\chi})$, etc. when there is no confusion. Note that $P_{1,0} = \mxtr\chi$ and $P_{0,1} = \mxtr\chib$, and a simple computation yields
\[T_{1, 0}^{ab}=\udl{T}_{0,1}^{ab}=\sigma^{ab}.\]
Moreover, when $n=3$, it is easy to check that
 \begin{align}\label{T in dimension 3}
\begin{split}
2T_{1,1}^{ab} &= \sigma^{ab} tr\udl{\chi}-\udl{\chi}^{ab};  \\
2 \udl{T}_{1,1}^{ab} &= \sigma^{ab} tr {\chi} -{\chi}^{ab}.  
\end{split}
\end{align}
It is worth remarking that the quantities $P_{r,r}$, \ $T_{r, r}^{ab} L$, \ $ \udl{T}_{r, r}^{ab} \udl{L}$, \  $T_{r+1, r}^{ab}$ and $\udl{T}_{r, r+1}^{ab}$ are all independent of the scaling of $L$ and $\udl{L}$ as $L \rw aL,\ \Lb \rw \frac{1}{a} \Lb$. 

\medskip
In the rest of this section, we focus on spacelike codimension-two submanifolds in a spacetime of constant curvature such as the Minkowski spacetime $\R^{n,1}$, the de-Sitter spacetime, or the anti de-Sitter spacetime. Before proving the Minkowski formulae for those mixed higher order mean curvatures, we observe that $P_{r,s}(\chi, \udl \chi)$ shares the divergence free property as $\sigma_k$.

\medskip
\begin{lem}
Let $\Sigma$ be a spacelike codimension-two submanifold in a spacetime of constant curvature. Suppose $\Sigma$ is torsion-free with respect to $L$ and $\Lb$. Then $T^{ab}_{r,s}$ and $\Tb^{ab}_{r,s}$ are divergence free for any $(r,s)$, that is, 
\begin{equation}\label{divergence_eq} \na_b T^{ab}_{r,s} = \na_b\Tb^{ab}_{r,s} =0. \end{equation}
\end{lem}
\begin{proof}
We denote $\tilde{\sigma}_{ab}= (\sigma+y \chi+\udl{y}\udl{\chi})_{ab}$ and its inverse by $(\tilde{\sigma}^{-1})^{ab}$. Formally differentiating \eqref{det} with respect to $\chi_{ab}$ and $\udl{\chi}_{ab}$, we obtain
\begin{eqnarray}\label{eq.T_ab}
y (\tilde{\sigma}^{-1})^{ab}\det(\tilde{\sigma})&=\sum\frac{(r+s)!}{r! s!} y^r \udl{y}^s T_{r,s}^{ab},
\end{eqnarray}
and
\begin{eqnarray}\label{eq.T_ab1}
\udl{y}(\tilde{\sigma}^{-1})^{ab}\det(\tilde{\sigma})&=\sum \frac{(r+s)!}{r! s!}y^r \udl{y}^s \udl{T}_{r,s}^{ab}.
\end{eqnarray}
\par
Next, taking a covariant derivative on both sides of equation (\ref{eq.T_ab}), the left-hand side becomes 
\begin{align*}
-y (\tilde{\sigma}^{-1})^{ac}\nabla_b(y\chi+\udl{y}\udl{\chi})_{cd} (\tilde{\sigma}^{-1})^{db}\det(\tilde{\sigma})
+y(\tilde{\sigma}^{-1})^{ab}\nabla_b(y\chi+\udl{y}\udl{\chi})_{cd} (\tilde{\sigma}^{-1})^{cd}\det(\tilde{\sigma}).
\end{align*}

Switching indices $b$ and $c$ in the second summand, we arrive at 

\begin{align}\label{divergence}
y \left(\tilde{\sigma}^{-1}\right)^{ac}\left[y\left(\nabla_c \chi_{bd}-\nabla_b\chi_{cd}\right)+\udl{y} (\nabla_c \udl{\chi}_{bd}-\nabla_b\udl{\chi}_{cd})\right] \left(\tilde{\sigma}^{-1}\right)^{db}\det(\tilde{\sigma})=\sum_{r,s} \frac{(r+s)!}{r! s!}y^r \udl{y}^s \nabla_b T_{r,s}^{ab}.
\end{align}

Similar computation applied to (\ref{eq.T_ab1}) yields
\begin{align}\label{divergence_}
 \udl{y}\left(\tilde{\sigma}^{-1}\right)^{ac}\left[y(\nabla_c \chi_{bd}-\nabla_b\chi_{cd})+\udl{y} (\nabla_c \udl{\chi}_{bd}-\nabla_b\udl{\chi}_{cd}) \right] \left(\tilde{\sigma}^{-1}\right)^{db}\det(\tilde{\sigma})=\sum_{r,s}\frac{(r+s)!}{r! s!} y^r \udl{y}^s \nabla_b\udl{T}_{r,s}^{ab}. 
\end{align} 
\par
For submanifolds in spacetime, the Codazzi equations give 
\begin{equation}\label{codazzi}\begin{split}\nabla_c \chi_{bd}-\nabla_b\chi_{cd}&=\langle \bar{R}(\partial_c, \partial_b)L, \partial_d\rangle+\zeta_b\chi_{cd}-\zeta_c\chi_{bd}\\ 
\nabla_c \udl{\chi}_{bd}-\nabla_b\udl{\chi}_{cd}&= \langle \bar{R}(\partial_c, \partial_b)\udl L, \partial_d\rangle-\zeta_b\udl{\chi}_{cd}+\zeta_c\udl{\chi}_{bd}. \end{split}
\end{equation}

Since $\Sigma$ is assumed to be torsion-free and the ambient space is of constant curvature, the left hand side of both (\ref{divergence}) and (\ref{divergence_}) are zero. Then, the assertion follows by comparing the coefficients of the term $y^r\udl y^s$.
\end{proof}
\bigskip


\bigskip
We now derive the Minkowski formulae for the mixed higher order mean curvatures of torsion-free submanifolds in constant curvature spacetimes.
\medskip
\begin{theorem}[Theorem D]\label{minkowski1}
Let $\Sigma$ be a closed spacelike codimension-two submanifold in a spacetime of constant curvature. Suppose $\Sigma$ is torsion-free
with respect to the null frame $L$ and $\Lb$. Then
\begin{align}\label{r-1ands}
2\int_\Sigma P_{r-1, s}(\chi,\udl\chi)\langle  L, \frac{\partial}{\partial t}\rangle d\mu +\frac{r+s}{n-(r+s)} \int_\Sigma P_{r, s}(\chi,\udl\chi)Q(L, \udl{L}) d\mu =0 
\end{align}
and
\begin{align}\label{rands-1}
2\int_\Sigma P_{r, s-1}(\chi,\udl\chi)\langle  \Lb, \frac{\partial}{\partial t}\rangle d\mu -\frac{r+s}{n-(r+s)}  \int_\Sigma P_{r, s}(\chi,\udl\chi) Q(L, \udl{L}) d\mu =0. 
\end{align}
\end{theorem}
\begin{proof}
By the divergence theorem, we have
\begin{equation}\label{nablazero}
\int_\Sigma \nabla_a[ T^{ab} Q(L, \partial_b)]d\mu=0 
\end{equation} 
where $T^{ab}$ is one of $T_{r,s}^{ab} $.
Since $\Sigma$ is torsion-free, direct computation shows \[ \nabla_a[ Q(L, \partial_b)]=(D_a Q)(L, \partial_b)+\chi_a^cQ_{cb}+\frac{1}{2}\chi_{ab} Q(L, \udl{L}).
\]
By (\ref{CKY equation}), we derive
\[ (D_a Q)(L, \partial_b)+(D_b Q)(L, \partial_a)=2\langle L, \frac{\partial}{\partial t}\rangle \sigma_{ab}.\]
Therefore, we obtain 
\[ \na_a [T^{ab}Q(L,\pl_b)] = T^{ab} \sigma_{ab} \langle L,\frac{\pl}{\pl t} \rangle + \frac{1}{2} (T^{ba}\chi_a^c-T^{ca}\chi_a^b)Q_{cb}
+\frac{1}{2} (T^{ab}\chi_{ab}) Q(L, \udl{L}),\] where the second term on the right hand side comes from anti-symmetrization of the indices. 
Recall the Ricci equation
\begin{equation}\label{ricci_eq}\frac{1}{2} \chi_a^{\;\;c} \chib_{cb} - \frac{1}{2} \chi_b^{\;\;c} \chib_{ca} + (d\zeta)_{ab} = \frac{1}{2} \langle \bar{R}(\pl_a,\pl_b)L,\Lb \rangle.\end{equation}
Since $\Sigma$ is torsion-free and the ambient space is of constant curvature, the equation implies $\chi$ and $\udl{\chi}$ commute. Note that $T^{ba}$ is a polynomial of $\chi$ and $\udl{\chi}$, and thus also commutes with  $\chi$ and $\udl{\chi}$. It follows that
\[
T^{ba}\chi_a^c-T^{ca}\chi_a^b=0.
\]
Putting these together, (\ref{nablazero}) implies 
\[
\int_{\Sigma} \left(\sigma_{ab} T^{ab}_{r,s}\right) \langle L, \frac{\partial}{\partial t}\rangle d\mu + \frac{1}{2} \int_{\Sigma} \left(\chi_{ab}T^{ab}_{r, s}\right)Q(L, \udl L) d\mu =0.
\]
\eqref{r-1ands} follows from \eqref{a} and \eqref{b} in the appendix.

The second formula is derived similarly by considering
\begin{eqnarray}\label{nablazero1}
 \int_\Sigma \na_a [\Tb^{ab}Q(\Lb,\pl_b)] d\mu =0, 
 \end{eqnarray}
and using
\[
\nabla_a \left[ Q(\udl L, \partial_b ) \right] = (D_a Q) (\udl L, \partial_b) + \udl \chi_a^c Q_{cb}- \frac{1}{2} \udl \chi_{ab}Q(L, \udl L).
\]
\end{proof}

\medskip
\par
From the above proof, we see that the spacetime Minkowski formulae (\ref{r-1ands}) and (\ref{rands-1}) follow from 
\[
\int_\Sigma \nabla_a[ T^{ab} Q(L, \partial_b)]d\mu=0, \text{ and }\int_\Sigma \na_a [\Tb^{ab}Q(\Lb,\pl_b)] d\mu =0.
\]
In fact, we have two more possible identities:
\[
\int_\Sigma \nabla_a[ T^{ab} Q(\udl L, \partial_b)]d\mu=0, \text{ and }\int_\Sigma \na_a [\Tb^{ab}Q(L,\pl_b)] d\mu =0.
\]
Following the same line, one can prove another two spacetime Minkowski formulae:
\begin{align}\label{r-1andsL}
2\int_\Sigma P_{r-1, s}(\chi,\udl\chi)\langle  \udl L, \frac{\partial}{\partial t}\rangle d\mu - \frac{r+s}{n-(r+s)} \int_\Sigma P_{r-1, s+1}(\chi,\udl\chi)Q(L, \udl{L}) d\mu =0, 
\end{align}
and 
\begin{align}\label{rands-1L}
2\int_\Sigma P_{r, s-1}(\chi,\udl\chi)\langle  L, \frac{\partial}{\partial t}\rangle d\mu + \frac{r+s}{n-(r+s)} \int_\Sigma P_{r+1, s-1}(\chi,\udl\chi)Q(L, \udl{L}) d\mu =0.
\end{align}

\bigskip
\par
We finish this section by showing that (\ref{r-1ands}) and (\ref{rands-1}) recover the classical Minkowski formulae for hypersurfaces in Riemannain space forms - Euclidean space, hemisphere, and hyperbolic space.

It is well known that these space forms $S^n_{-\kappa}$ can be embedded as totally geodesic time-slices in the Minkowski spacetime, the de Sitter spacetime, and the anti-de Sitter spacetime respectively. We write the spacetime metric in static coordinates:
\[ \bar{g} = - (1+\kappa r^2) dt^2 + \frac{1}{1+\kappa r^2} dr^2 + r^2 g_{S^{n-1}}. \] 
Given a hypersurface $\Sigma \subset S^n_{-\kappa}$, we view it as a spacelike codimension-two submanifold lying in a totally geodesic time-slice. Let $\nu$ denote the outward unit normal of $\Sigma$ in the totally geodesic slice. We take $L = \frac{1}{\sqrt{1+\kappa r^2}}\frac{\pl}{\pl t} + \nu$ and $\Lb = \frac{1}{\sqrt{1+\kappa r^2}}\frac{\pl}{\pl t} - \nu$ to get $\chi = -\chib = h$ and $\zeta=0$. Here $h$ is the second fundamental form of $\Sigma$ in the totally geodesic time-slice. Therefore, according to the definition (\ref{det}), 
\[ 
\det(\sigma + (y-\udl{y})h)  = \sum_{0 \le r+s \le n-1} \frac{(r+s)!}{r!s!}y^r \udl{y}^s P_{r,s}(\chi, \udl \chi)
\]
and thus $P_{r,s}(\chi, \udl \chi)= (-1)^s \sigma_{r+s}(h)$. Moreover, $\langle L, \frac{\pl}{\pl t} \rangle = -\sqrt{1+\kappa r^2}$ and standard computation gives
\begin{eqnarray}\label{QLL}
Q(L,\Lb) = 2 \langle X,\nu \rangle, \ \text{ with } X=r \sqrt{1+\kappa r^2} \frac{\pl}{\pl r}.
\end{eqnarray}
\par
Putting these expressions into the spacetime Minkowski formula (\ref{Minkowski formula for higher order mean curvatureL}), it reduces to
\begin{align}\label{recover}
 (n-r-s) \int_\Sigma \sqrt{1+\kappa r^2} \cdot \sigma_{r+s-1} (h)d\mu = (r+s) \int_\Sigma \sigma_{r+s} (h)\cdot \langle X,\nu \rangle d\mu.
\end{align} 
\par
 We thus recover the classical Minkowski formula (\ref{Minkowski.formula}) in the Euclidean space by letting $k=r+s$ and $\kappa=0$. For the remaining cases, we view the hemisphere as the hypersurface defined by \[ (x^0)^2 + \cdots + (x^n)^2=1, x^0 >0\]
in the Euclidean space and the hyperbolic space as the hypersurface defined by 
\[ 
-(x^0)^2 + (x^1)^2 + \cdots + (x^n)^2=-1, \ x^0 >0
\] 
in the Minkowski spacetime. After making a change of variable $r = \sin\theta$ ($r = \sinh\theta$, resp.) and noting that $X$ is the tangential component of $- \frac{\pl}{\pl x^0}$ ($\frac{\pl}{\pl x^0}$, resp.) for hemisphere (hyperbolic space, resp.), it is not difficult to see that (\ref{recover}) recovers \cite[Corollary 3(b), 3(c)]{Bivens83} (see also \cite{GuanLi, Montiel-Ros91, Strubing84}).


\medskip
\bigskip
\section{Alexandrov Theorems for Submanifolds of constant mixed higher order mean curvature in a spacetime of constant curvature}
\medskip
\par
In Section 3, the simplest case of the spacetime Minkowski formula (Theorem A) was applied to establish the spacetime Alexandrov type theorems concerning codimension-two submanifolds with $\langle \vec{H}, L\rangle =constant$. It is interesting to replace the mean curvature by other invariants from the second fundamental form. 
In the hypersurface case, Ros \cite{Ros87} showed that any closed, embedded hypersurface in $\mathbb R^n$ with constant $\sigma_k$ curvature is a round sphere. This result was generalized to the hyperbolic space by Montiel and Ros \cite{Montiel-Ros91} and to the Schwarzschild manifold by Brendle and Eichmair \cite{BE}. 
\par
In this section, we consider codimension-two submanifolds in a spacetime of constant curvature. More precisely, using the spacetime Minkowski formulae established in the previous section, we prove two Alexandrov-type theorems for submanifolds of constant mixed higher order mean curvatures. The first one assumes constancy of $P_{r,0}$ or $P_{0, s}$ and concludes that the submanifold lies in a null hypersurface of symmetry. The second one assumes the stronger condition of constancy of $P_{r,s}$ for $r>0, s>0$, which forces $\Sigma$ to be a sphere of symmetry. 

\medskip
To state our first Alexandrov theorem, we recall the definition of $\Gamma_k$ cone (also see Definition \ref{Gamma_k cone} in Appendix A). For $1\leq k\leq n-1$, $\Gamma_k$ is a convex cone in $\mathbb R^{n-1}$ such that 
$\Gamma_k = \{ \lambda \in \mathbb R^{n-1} \ : \ \sigma_1(\lambda)>0, \cdots, \sigma_k(\lambda) >0 \}$ where \[\sigma_k(\lambda) = \sum_{i_1<\cdots< i_k} \lambda_{i_1} \cdots \lambda_{i_k}\] is the $k$-th elementary symmetric function. An $(n-1)\times(n-1)$ symmetric matrix $W$ is said to belong to $\Gamma_k$ if its spectrum $\lambda(W) \in \Gamma_k$.

\medskip
 \begin{theorem}\label{Alexandrov theorem for Weingarten submanifolds in Minkowski spacetime}
Let $\Sigma$ be a past (future, respectively) incoming null embedded, closed spacelike codimension-two submanifold in an $(n+1)$-dimensional
spacetime of constant curvature. Suppose $\Sigma$ is torsion-free with respect to $L$ and $\Lb$ and the second fundamental form $\chi \in \Gamma_r$ ($-\chib \in \Gamma_s$, respectively). If $P_{r,0} = C$ ($P_{0,s} = (-1)^s C$, respectively) for some positive constant $C$ on $\Sigma$, then $\Sigma$ lies in a null hypersurface of symmetry.
\end{theorem}
\begin{proof}
By the assumption that $P_{r, 0}=C$, the Minkowski formula (\ref{Minkowski formula for higher order mean curvatureL}) becomes
\[ \int_\Sigma \frac{P_{r-1,0}}{P_{r,0}} \langle L, \frac{\pl}{\pl t} \rangle d\mu + \frac{r}{2(n-r)} \int_\Sigma Q(L,\Lb) d\mu =0. \] 
Applying the Newton-Maclaurin inequality (\ref{NM}) repeatedly and noting that $\langle L ,\frac{\pl}{\pl t} \rangle <0$, we obtain
\[ \int_\Sigma \frac{\langle L,\frac{\pl}{\pl t} \rangle}{\mxtr\chi} d\mu +  \frac{1}{2(n-1)} \int_\Sigma Q(L,\Lb) d\mu \ge 0. \]
Comparing this with the spacetime Heintze-Karcher inequality (\ref{Heintze-Karcher inequality1}), we see that the equality is achieved. Theorem \ref{Spacetime Heintze-Karcher inequality: past incoming} implies that $\Sigma$ lies in a null hypersurface of symmetry.

For the corresponding statement for $\chib$, note that under the assumption the Minkowski formula becomes
\[ \int_\Sigma \frac{P_{0,s-1}}{P_{0,s}} \langle \Lb,\frac{\pl}{\pl t} \rangle d\mu -  \frac{s}{2(n-s)} \int_\Sigma Q(L,\Lb) d\mu =0. \]
Applying the Newton-Maclaurin inequality repeatedly, we obtain
\[ \frac{P_{0,s-1}(\chi,\chib)}{P_{0,s}(\chi,\chib)} \langle \Lb, \frac{\pl}{\pl t} \rangle = -\frac{P_{0,s-1}(\chi,-\chib)}{P_{0,s}(\chi,-\chib)} \langle \Lb,\frac{\pl}{\pl t} \rangle \ge - \frac{s(n-1)}{n-s} \frac{1}{P_{0,1}(\chi,-\chib)} \langle \Lb, \frac{\pl}{\pl t} \rangle. \]
As ${P_{0,1}(\chi,-\chib)}=\langle \vec{H}, \udl L\rangle$, the equality of the spacetime Heintze-Karcher inequality is achieved again and $\Sigma$ lies in a null hypersurface of symmetry by Theorem \ref{Spacetime Heintze-Karcher inequality}.
\end{proof}

\bigskip
\par
In the rest of this section, we prove a rigidity result for submanifolds with constant $P_{r,s}$ for $r>0, s>0$. We start with an algebraic lemma.

%

\begin{lemma}\label{lemma}
Suppose $\chi, \chib \in \Gamma_{r+s}$. If
\begin{eqnarray}\label{conditiontech}
\frac{\chib_{ab} (\Tb_{0,s})^{bc} }{ P_{0,s} } \frac{\chi_c^{\;\;a}}{P_{1,0}} \geq \frac{s }{n-1},
\end{eqnarray}
then $\chi$ and $\chib$ satisfy the following inequality
 \begin{eqnarray}\label{inequality0}
 \frac{P_{r-1, s}(\chi, \udl\chi)}{ P_{r, s}(\chi, \udl\chi)} \geq \frac{r+s}{n-(r+s)}\frac{n-1}{tr \chi}.
 \end{eqnarray}
The equality holds if and only if $\chi$ is a multiple of the identity matrix.
\end{lemma}

\begin{proof}
Since $\chi \in \Gamma_{r+s}$ and $\chib \in  \Gamma_{r+s}$, $P_{r',s}(\chi, \udl\chi)>0$ for any $r'$ with $1 \le r' \le r$ by \eqref{positivity} and the fact that $\Gamma_{r+s}\subset\Gamma_{r'+s}$. We apply Newton-MacLaurin inequality (\ref{NM}) repeatedly to get 
\begin{eqnarray*}\label{inequality2}
\frac{\left(n-(r+s)\right) P_{r-1, s} (\chi, \udl\chi)}{(r+s)P_{r, s }(\chi, \udl\chi)} \geq \frac{(n-1-s)}{s+1}\frac{P_{0, s}(\chi, \udl\chi)}{P_{1, s}(\chi, \udl\chi)}.
\end{eqnarray*}
It suffices to show that  
\begin{eqnarray}\label{inequality3}
\frac{(n-1-s)}{s+1}\frac{P_{0, s}(\chi, \udl\chi)}{P_{1, s}(\chi, \udl\chi)} \geq \frac{n-1}{tr\chi}.
\end{eqnarray}
Let $\chib_i$ ($i=1,2,\cdots, n-1$) be the eigenvalues of $\chib$. 
By the definition of completely polarized elementary symmetric function in \eqref{polar2} and \eqref{polar3}, 
\begin{eqnarray*}
P_{0, s}(\chi, \udl\chi) = \sigma_{(s)}(\udl\chi, \cdots, \udl\chi) =: \sigma_s(\udl\chi),
\end{eqnarray*}
and
\[
P_{1, s}(\chi, \udl\chi) =\sigma_{(s+1)}(\chi, \underbrace{\udl\chi, \cdots, \udl\chi}_s) = \frac{1}{s+1} \sum_{a, b}\chi_{ab} \frac{\partial \sigma_{s+1}(\udl\chi)}{\partial \udl\chi_{ab}} =\frac{1}{s+1} \sum_{i}^{n-1}\chi_{ii} \sigma_{s}(\udl\chi | i ).
\]
Here $\chi_{ii}$ is the $(i, i)$ entry of matrix $\chi$, $\sigma_s(\chib | i) = \sum \chib_{j_1} \cdots \chib_{j_s}$ and we sum over distinct $j_k \neq i, j_k = 1, \ldots, n-1$.
\par
On the other hand, the assumption (\ref{conditiontech}) is equivalent to 
\[\frac{1}{\sigma_s(\chib)} \sum_{i=1}^{n-1} \chib_i \sigma_{s-1}(\chib | i) \chi_{ii} \ge \frac{s}{n-1}  \mxtr\chi. \] 
Hence
\begin{align*}
\sum_{i=1}^{n-1} \chi_{ii} \sigma_s(\chib|i) &=  \sigma_s(\chib) \mxtr\chi - \sum_{i=1}^{n-1} \chi_{ii} \chib_i \sigma_{s-1}(\chib | i) \le \frac{n-s-1}{n-1}  \sigma_s(\chib)\mxtr\chi.
\end{align*}
\par
Again, by the cone condition on $\chi$ and $\udl\chi$, we have $\mxtr\chi > 0$, $\sum \chi_{ii} \sigma_s(\chib |i) = (s+1)P_{1,s} >0$, and thus
\begin{eqnarray}\label{inequality4}
\frac{(n-1-s)}{s+1}\frac{P_{0, s}(\chi, \udl\chi)}{P_{1, s}(\chi, \udl\chi)} = \frac{(n-1-s)\sigma_s(\udl\chi)}{ \sum_{i}^{n-1}\chi_{ii} \sigma_{s}(\udl\chi | i )} \ge \frac{n-1}{\mxtr\chi}, 
\end{eqnarray}
which gives us the desired inequality. Moreover, by tracing back the proof, we notice that the equality in (\ref{inequality0}) holds only if this is the case in both the Newton-MacLaurin inequality and \eqref{conditiontech}. The former one tells us that $\chi$ is a multiple of $I_{n-1}$. And this also implies the equality for (\ref{conditiontech}) by the elementary identity $\sum_i \udl\chi_i \sigma_{s-1}\left(\udl\chi | i\right) = s\sigma_s\left(\udl\chi\right)$. On the other hand, it is easy to see that if $\chi$ is a multiple of $I_{n-1}$, the equality in (\ref{inequality0}) is achieved.
\end{proof}

\medskip
Before moving to the Alexandrov type theorem, we make a remark on the conditions in the previous algebraic lemma.
\begin{remark}
The above algebraic lemma still holds if we replace the cone condition by $\chi\in \Gamma_{r+s}$ and $-\chib \in \Gamma_{r+s}$, since the left hand sides of both (\ref{conditiontech}) and (\ref{inequality0}) are homogeneous of degree zero in $\udl\chi$. 
\end{remark}
\begin{remark}
The technical condition (\ref{conditiontech}) can be interpreted as follows. Consider $\vec{u}=(u_1, \cdots, u_{n-1})$ with $u_i = \frac{\chib_i \sigma_{s-1}(\chib | i) }{\sigma_s(\chib)} $ as a vector determined by $\udl\chi$ and $\vec{v}= (v_1, \cdots, v_{n-1})$ with $v_i=\frac{\chi_{ii}}{\sigma_1(\chi)}$ as a vector determined by $\chi$, (\ref{conditiontech}) imposes a restriction on the angle between the vectors $\vec{u}$ and $\vec{v}$.
\end{remark}

\medskip

\begin{theorem}\label{Alexandrov for mixed higher mean curvature}
Let $\Sigma$ be a past incoming null embedded (see Definition \ref{future incoming null embedded}) closed embedded spacelike codimension-two submanifold in a spacetime of constant curvature. Suppose $\Sigma$ is torsion-free with respect to the null frame $L$ and $\Lb$ and that the second fundamental forms $\chi\in \Gamma_{r+s}$ and $-\udl\chi \in \Gamma_{r+s}$ satisfy
\begin{eqnarray}\label{constant0}
P_{r, s}(\chi, \udl\chi)= (-1)^sC, \textit{where } C \textit{ is a positive constant. }
\end{eqnarray}
and 
\begin{eqnarray}\label{condition}
\frac{\chib_{ab} (\Tb_{0,s})^{bc} }{ P_{0,s} } \chi_c^{\;\;a} \geq \frac{s }{n-1}P_{1,0}.
\end{eqnarray}
Then $\Sigma$ is a sphere of symmetry.
\end{theorem}
\begin{proof}
For submanifolds with $P_{r,s} = constant$, the Minkowski formula (\ref{Minkowski formula for higher order mean curvatureL}) becomes
\begin{eqnarray*}
-\int_\Sigma \frac{P_{r-1, s}}{P_{r, s}}\langle  L, \frac{\partial}{\partial t}\rangle d\mu =\frac{1}{2}\frac{r+s}{n-(r+s)} \int_\Sigma Q(L, \udl{L}) d\mu. 
\end{eqnarray*}
It follows from Lemma \ref{lemma} and the fact that $-\langle  L, \frac{\partial}{\partial t}\rangle\geq 0$,
\begin{eqnarray*}
-(n-1)\int_\Sigma \frac{\langle  L, \frac{\partial}{\partial t}\rangle}{tr\chi} d\mu \leq \frac{1}{2} \int_\Sigma Q(L, \udl{L}) d\mu. 
\end{eqnarray*}
This together with the spacetime Heintze-Karcher inequality (Theorem \ref{Spacetime Heintze-Karcher inequality: past incoming}) imply 
\begin{eqnarray}\label{equality}
(n-1)\int_\Sigma \frac{\langle  L, \frac{\partial}{\partial t}\rangle}{\langle\vec{H}, L\rangle} d\mu \leq \frac{1}{2} \int_\Sigma Q(L, \udl{L}) d\mu \leq (n-1)\int_\Sigma \frac{\langle  L, \frac{\partial}{\partial t}\rangle}{\langle\vec{H}, L\rangle} d\mu.
\end{eqnarray}
Thus, the equality must hold. Note that the first inequality comes from the Newton-MacLaurin inequality (\ref{NM}) and the equality case implies $\chi=c_1 \sigma$. And the second inequality is from the Heintze-Karcher inequality and equality holds if $\chi = c_1 \sigma$.
\par
On the other hand, equation (\ref{constant0}) together with $\chi=c_1\sigma$ imply that 
\[
P_{0, s}(\chi, \udl\chi)= (-1)^s \tilde{C}, \textit{ where } \tilde{C} \textit{ is a positive constant}.
\]
which falls in the setting of Theorem \ref{Alexandrov theorem for Weingarten submanifolds in Minkowski spacetime}. Follow the same line of the proof there and apply Minkowski formula
\[ 
\int_\Sigma \frac{P_{0,s-1}}{P_{0,s}} \langle \Lb,\frac{\pl}{\pl t} \rangle d\mu -  \frac{s}{2(n-s)} \int_\Sigma Q(L,\Lb) d\mu =0,
\]
we can arrive the equality case for the Heintze-Karcher inequality (\ref{Heintze-Karcher inequality}) by using the Newton-MacLaurin inequality. Then, we conclude that $\udl\chi=c_2 \sigma$.
\end{proof}
\medskip

\par
As remarked above, \eqref{condition} seems to be a technical condition. However, we believe that, certain reasonable condition on $\chi$ and $\udl\chi$ in addition to $P_{r,s}(\chi, \udl\chi)=C$ is necessary in order to conclude both of them are proportional to $\sigma$. 

\par
To finish this section, we present two settings in which condition (\ref{condition}) is automatically satisfied. 
\par
The first example is: $\udl\chi= -\chi$. From the discussion at the end of Section 4, we see that the classical hypersurfaces cases fall in this setting. In this special situation, (\ref{condition}) is equivalent to 
 \begin{eqnarray*}
1&\leq& \frac{ (n-1) \sum_{i=1}^{n-1} \sigma_{s-1} \left((-\chi) | i\right) (-\chi)_i \chi_i}{ s \sigma_1(\chi) \sigma_s(-\chi)}=  \frac{ (n-1) \sum_{i=1}^{n-1} \sigma_{s-1} (\chi | i) \chi_i^2}{ s \sigma_1(\chi) \sigma_s(\chi)}\\
&=& \frac{(n-1) \left[ \sigma_s(\chi)\sigma_1(\chi) - (s+1) \sigma_{s+1}(\chi)\right]}{s\sigma_1(\chi)\sigma_s(\chi)}.
 \end{eqnarray*}
where $\chi_i$ denote the eigenvalues of $\chi$. However, this inequality follows from the standard Newton-MacLaurin inequality for symmetric functions: $(n-s-1) \sigma_{s}(\chi)\sigma_1(\chi) \geq (n-1)(s+1)\sigma_{s+1}(\chi)$.
In view of the remark at the end of the previous section, Theorem \ref{Alexandrov for mixed higher mean curvature} generalizes the classical Alexandrov theorem in Riemannian space forms \cite[Theorem 7 and Theorem 10]{Montiel-Ros91}.
\par
The second example is: one of $\chi$ and $\udl\chi$ is already known to be a multiple of $I_{n-1}$. One can easily check that (\ref{condition}) is trivial by using the elementary formula $\sum_{i}\udl\chi_i \sigma_{s-1}(\udl\chi | i) = s \sigma_s(\udl\chi)$.

\bigskip
\section{Generalization in the Schwarzschild spacetime}

In this section, we discuss Minkowski type formulae and Alexandrov theorems in the Schwarzschild spacetime. The divergence equations \eqref{divergence_eq} of $T^{ab}_{r, s}$ and $\udl T^{ab}_{r, s}$ play crucial roles in the proof of the Minkowski formula in a spacetime of constant
curvature. Those equations no longer hold in the Schwarzschild spacetime due to the presence of a non-trivial ambient curvature contribution. However, it turns out  the divergences of $T^{ab}_{r, s}$ and $\udl T^{ab}_{r, s}$ still posses  favorable properties under natural assumptions on $\Sigma$ when either $r=0$ or $s=0$..

\begin{lemma}\label{divergence structure}
Let $\Sigma$ be a spacelike codimension-two submanifold in the Schwarzschild spacetime. Suppose $\Sigma$ is torsion-free with respect to a null frame $L, \udl{L}$. Then the following statements are true:
\begin{enumerate}
\item[(1)] If $Q(L,\Lb) \ge 0$, then $\sum_{a, b} (\nabla_b T^{a b}_{2, 0}) Q(L, \partial_a)  \le  0$ and $ \sum_{a, b} (\nabla_b T^{a b}_{0,2}) Q(\Lb, \partial_a)   \le 0$.
\item[(2)] Suppose $\chi$ is positive definite and $(Q^2)(L,v) Q(L,v)  \le 0$ for any vector $v$ tangent to $\Sigma$, then $\sum_{a, b} (\nabla_b T^{a b}_{r, 0}) Q(L, \partial_a)   \le  0$ if $r\ge 3$. 
\item[(3)] Suppose $-\chib$  is positive definite and  $(Q^2)(\Lb,v) Q(\Lb,v)   \ge 0$ for any vector $v$ tangent to $\Sigma$, then
$\sum_{a, b} (\nabla_b T^{a b}_{0,s}) Q(\Lb, \partial_a)   \le  0$ if $s\ge 3$. 
\end{enumerate}

\end{lemma}

\begin{proof}

Denote the radial coordinate in the Schwarzschild metric by $\rho$ and write the metric as
\[
\overline{g}= - \left( 1- \frac{2m}{\rho^{n-2}}\right) dt^2 + \frac{1}{1- \frac{2m}{\rho^{n-2}}}d\rho^2 + \rho^2 g_{\mathbb S^{n-1}}.
\]
We only deal with case (2), and the other cases can be derived by the same argument.
 Denote $\tilde{\sigma} = \sigma + y\chi$ and write $T_r$ for $T_{r,0}$. Setting $\underline{y}=0$ in \eqref{divergence} and \eqref{codazzi}, we obtain
\begin{eqnarray*} \sum_r y^r \na_b  T^{ab}_r = y^2 \bar{R}_{Lbdc} (\tilde{\sigma}^{-1})^{ac} (\tilde{\sigma}^{-1})^{db} \det(\tilde{\sigma}) .
\end{eqnarray*} 
\par
From the curvature expression (\ref{curvature tensor in terms of Q: Schwarzschild}), we have 
\[ \bar{R}_{Lbdc} = -\frac{n(n-1)m}{\rho^{n+2}} \lt( \frac{2}{3} Q_{Lb}Q_{dc} - \frac{1}{3} Q_{Ld}Q_{cb} - \frac{1}{3} Q_{Lc} Q_{bd} \rt) - \frac{nm}{\rho^{n+2}} \lt( (Q^2)_{Ld} \sigma_{bc} - (Q^2)_{Lc} \sigma_{bd} \rt),\]
and thus 
\begin{align*} 
&\sum_r y^r \na_b  T^{ab}_r Q_{La} \\
&= -\frac{nm}{\rho^{n+2}} \lt[ (n-1)Q_{Lb}Q_{dc} + (Q^2)_{Ld} \sigma_{bc} - (Q^2)_{Lc} \sigma_{bd}\rt] Q_{La} (\tilde{\sigma}^{-1})^{ac} (\tilde{\sigma}^{-1})^{db} \det(\tilde{\sigma}).
\end{align*}
Now, suppose that $\chi$ is diagonal with eigenvalues $\chi_1, \ldots, \chi_{n-1}$. Write the eigenvalues of $\tilde{\sigma}$ as $\mu_a = 1 + y\chi_a$.  We obtain
\begin{eqnarray*}
\sum_r y^r \na_b  T^{ab}_r Q_{La} &= \frac{nm}{\rho^{n+2}} y^2\lt[ \sum_{a \neq b} \frac{\sigma_{n-1}(\tilde{\sigma})}{\mu_a \mu_b} (Q^2)_{La} Q_{La} \rt].
\end{eqnarray*}

On the other hand, by standard computation, we have
\begin{eqnarray*}
\sum_{a\neq b} \frac{\sigma_{n-1}(\tilde\sigma)}{\mu_a\mu_b} &=& \sum_{a\neq b} \sigma_{n-3} (\mu\ | \ a b)=\sum_{a\neq b}\sum_{q=0}^{n-3} y^q \sigma_{q} (\chi \ | \ ab)= \sum_{q=0}^{n-3} y^q\sum_{a=1}^{n-1} (n-q-2)\sigma_q (\chi | a).
\end{eqnarray*}
To get the last equality, we use the property of elementary symmetric function that $\sum_{i=1}^m \sigma_k(\lambda \ | \ i) = (m-k) \sigma_k(\lambda)$.
Thus, 
\begin{eqnarray}\label{gradientT}
\sum_r y^r \na_b  T^{ab}_r Q_{La} &=&  \frac{nm}{\rho^{n+2}} \sum_{a=1}^{n-1} \sum_{q=0}^{n-3} y^{q+2} (n-q-2) \sigma_{q}(\chi | a) (Q^2)_{La} Q_{La}\\\nonumber
&=& \frac{nm}{\rho^{n+2}} \sum_{a=1}^{n-1} \sum_{p=2}^{n-3} y^{p} (n-p) \sigma_{p-2}(\chi | a) (Q^2)_{La} Q_{La}
\end{eqnarray}
\par
By comparing the coefficients of $y^r$, we obtain

\begin{equation}\label{condition_on_Q} \sum_{a, b}(\na_b  T^{ab}_r) Q_{La} = \frac{nm(n-r)}{\rho^{n+2}} \sum_{a=1}^{n-1}   \sigma_{r-2}(\chi | a) (Q^2)_{La} Q_{La}, \text{ for } r\ge 2,\end{equation}
which is negative by the assumptions that $\chi$ is positive definite and $(Q^2)(L,v) Q(L,v)  \le 0$ for any vector $v$ tangent to $\Sigma$. This proves the second statement. The third one is proved along exactly the same line.
\par
For $r=2$ (or $s=2$) case, by comparing the coefficients of $y^2$ on both sides of (\ref{gradientT}), we get
\begin{eqnarray*}
\sum_{a, b}(\na_b T^{ab}_2) Q_{La} &=& \frac{n(n-2)m}{\rho^{n+2}} \sum_{a=1}^{n-1} (Q^2)_{La}Q_{La}\\ 
&=&\frac{n(n-2)m}{\rho^{n+2}} \left[\sum_{a,c} Q_L^{\;\;c} Q_{ca} Q_{La} - \frac{1}{2} Q_{L\Lb} \sum_a (Q_{La})^2 \right]\\
&=& -\frac{1}{2}\frac{n(n-2)m}{\rho^{n+2}} Q_{L\Lb} \sum_a (Q_{La})^2,
\end{eqnarray*} 
which is non-positive by the assumption that $Q(L, \Lb)\ge 0$.
\end{proof}

\medskip
\par
Notice that no condition is needed for $r=1$ or $s=1$ since $T_{1,0}^{ab} =T_{0,1}^{ab} =\sigma^{ab}$ is always divergence free. Thus, we can prove a clean Minkowski formula for $(r, s)=(1, 0)$ or $(0, 1)$ in the Schwarzschild spacetime, see Theorem A or Theorem \ref{first Minkowski identity}. For $r, s\geq 2$, the divergence property of $T_{r, 0}$ and $T_{0, s}$ no longer holds. Fortunately, based on the above lemma, we can still establish certain inequalities for those higher order cases in the Schwarzschild spacetime. 

\begin{theorem}\label{Minkowski in Schwarzschild}
Let $\Sigma$ be a closed spacelike codimension-two submanifold in the Schwarzschild spacetime. Suppose that $\Sigma$ is torsion-free.
\par

\par
If $\Sigma$ satisfies assumption (1) or (2) in Lemma \ref{divergence structure}, then for any $1\leq r\leq n-1$,
\begin{eqnarray}\label{Minkowski inequality1}
\int_\Sigma P_{r-1, 0}(\chi,\udl\chi)\langle  L, \frac{\partial}{\partial t}\rangle d\mu+\frac{r}{2(n-r)} \int_\Sigma P_{r, 0}(\chi,\udl\chi)Q(L, \udl{L})d\mu\geq 0.
\end{eqnarray}

\par
If $\Sigma$ satisfies assumption (1) or (3) in Lemma \ref{divergence structure}, then for any $1 \le s \le n-1$, 
\begin{eqnarray}\label{Minkowski inequality2}
 \int_\Sigma P_{0, s-1}(\chi,\udl\chi)\langle  \Lb, \frac{\partial}{\partial t}\rangle d\mu -\frac{s}{2(n-s)}  \int_\Sigma P_{0, s}(\chi,\udl\chi) Q(L, \udl{L})d\mu\geq 0. 
 \end{eqnarray}

\end{theorem}
\begin{proof}
Note that $T^{ab}_{r, 0}$ is a polynomial in $\chi$ only and thus $T^{ba}_{r,0} \chi_a^c - T^{ca}_{r,0}\chi_a^b=0$.  By the above lemma, we can proceed as in Theorem \ref{minkowski1}.
\end{proof}
In \cite{BE}, Brendle and Eichmair considered the case that $\Sigma$ is a closed embedded hypersurface contained in a totally geodesic time-slice $M$ of the Schwarzschild spacetime. Assume $\Sigma$ to be star-shaped and convex, they derived an interesting integral inequality 
\begin{equation}\label{Minkowski inequality3}
(n-k) \int_\Sigma f\sigma_{k-1} d\mu \leq k\int_\Sigma \sigma_k \langle X,\nu \rangle d\mu
\end{equation} 
where $f=\sqrt{1-\frac{2m}{\rho^{n-2}}}$, $X = \rho f \frac{\pl}{\pl \rho}$ is the conformal Killing vector and $\nu$ is the outward unit normal vector field of $\Sigma$. 
\par
In fact, \eqref{Minkowski inequality3} can be recovered by \eqref{Minkowski inequality1} or \eqref{Minkowski inequality2}. The argument is similar as the discussion at the end of section 4 where we recover the classical Minkowski formulae \eqref{Minkowski.formula} by \eqref{Minkowski formula for higher order mean curvatureL} or \eqref{Minkowski formula for higher order mean curvatureLb} except that one needs to check the assumption in Theorem \ref{Minkowski in Schwarzschild}. First, it is easy to see that being star-shaped is equivalent to $Q(L, \udl L)\ge 0$ because of identity \eqref{QLL}. On the other hand, the convexity of $\Sigma$ implies $\chi$ is positive definite. 
\par

For a submanifold $\Sigma$ on a totally geodesic slice $M_t$, the term \eqref{condition_on_Q} can be compared with the Ricci curvature term in Brendle-Eichmair's formula \cite{BE} (at the end of the proof of Proposition 8).  Indeed, given two vectors $v,w$ tangent to $M_t$, the Ricci curvature satisfies
\begin{align*}
&Ric_{M_t}(v,w) = R(v,e_{n+1},w,e_{n+1}) \\
&= -\frac{2m}{\rho^n} \bar{g}(v,w) - \frac{n(n-1)m}{\rho^{n+2}} Q(v,e_{n+1})Q(w,e_{n+1})- \frac{nm}{\rho^{n+1}} \lt( \bar{g}(v,w) Q^2(e_{n+1},e_{n+1}) - Q^2(v,w) \rt), 
\end{align*}
by the Gauss equation and (\ref{curvature tensor in terms of Q: Schwarzschild}). Let $\nu$ be the outward normal of $\Sigma$. We note that
\begin{align*}
Q^2(e_i,\nu) = Q^2(L,e_i) =  Q(\nu,e_{n+1})Q(e_i, e_{n+1}),
\end{align*} 
where $L = e_{n+1} + \nu $. Hence, $Ric_{M_t}(e_i,\nu) = -\frac{n^2m}{\rho^{n+2}} Q^2(L,e_i)$. 
\par

On the other hand, we claim that the condition in Lemma \ref{divergence structure}
\[ 
(Q^2)(L,v) Q(L,v)  \le 0 \ \text{ for any vector } v \text{ tangent to } \Sigma
\]
is automatically satisfied under the star-shaped condition. Indeed, the main ingredient is that the tangent vector $v$ does not have $\frac{\partial}{\partial t}$ component if $\Sigma$ lies in a totally geodesic time-slice. By the definition of $Q^2$,  \eqref{Q^2}, and noting that $Q(\partial_b, v)= \rho d\rho \wedge dt (\partial_b, v)=0$ for any tangent vector $\partial_b$, we expand
\begin{align*}
(Q^2)(L,v) = \bar{g}^{\udl L L}Q( L, \udl L)Q(L, v).
\end{align*}
Therefore,
\[
(Q^2)(L,v) Q(L,v) = -\frac{1}{2} Q(L, \udl L) \left(Q(L, v)\right)^2 \le 0,
\]
provided that $\Sigma$ is star-shaped.

\medskip
\par

Again, once we have the Minkowski formulae at hand, the spacetime Alexandrov theorems follow by the spacetime Heintze-Karcher inequality as in Theorem \ref{Alexandrov theorem for Weingarten submanifolds in Minkowski spacetime}.

\begin{corollary}\label{Alexandrov in Schwarzschild}
Let $\Sigma$ be a past (future, respectively) incoming null embedded (see Definition \ref{future incoming null embedded}) closed spacelike codimension-two submanifold in the Schwarzschild spacetime.  Suppose that $\Sigma$ is torsion-free. If $\Sigma$  satisfies the assumptions in either (1) or (2) ( (1) or (3), respectively) in Lemma \ref{divergence structure} and 
\begin{eqnarray}\label{constant}
P_{r, 0}(\chi, \udl\chi)= C  \hspace{1cm}\left(  P_{0,s}(\chi,\chib) = (-1)^s C , \mbox{respectively} \right)
\end{eqnarray}
for some positive constant $C$, then $\Sigma$ lies in a null hypersurface of symmetry.
\end{corollary}

\begin{remark}
In \cite{Li}, Li-Wei-Xiong show that the convexity assumption in Brendle-Eichmair's result can be removed. The same argument works here if we assume $\chi$ (or $\udl\chi$) is positive definite at a point. \end{remark} 

\bigskip
The above discussion on the Schwarzschild spacetime indicates that it is not easy to get a clean general form of Minkowski formulae with nontrivial curvature and torsion. In the rest of this section, we focus on a closed spacelike 2-surface in the 4-dimensional Schwarzschild spacetime which carries the two-form $Q=rdr\wedge dt$. With the same notations as in Section 4, we have
\[
T^{ab}_{2, 0}=  (\mxtr\chi) \sigma^{ab} - \chi^{ab}, \text{ and } \ \udl T^{ab}_{0,2} = (\mxtr\chib) \sigma^{ab} - \chib^{ab}.
\]
Recall the Codazzi equations
\begin{align}
\na_a \chi_{bc} - \na_b \chi_{ac} = \bar{R}_{abcL} + \chi_{ac} \zeta_b - \chi_{bc} \zeta_a \label{Codazzi null L};\\
\na_a \chib_{bc} - \na_b \chib_{ac} = \bar{R}_{abc\Lb} - \chib_{ac}\zeta_b + \chib_{bc} \zeta_a \label{Codazzi null Lbar} .
\end{align}
Taking trace of (\ref{Codazzi null L}), (\ref{Codazzi null Lbar}), and using the Ricci-flatness of 4-dimensional Schwarzschild spacetime, we get
\begin{align}
\na_a T^{ab}_{2,0} = - \sigma^{ac} \bar{R}_{a\;\;cL}^{\;\;b}  - T^{ab}_{2,0} \zeta_a  = -\frac{1}{2} \bar{R}_{L\;\;\Lb L}^{\;\;b} - T^{ab}_{2,0} \zeta_a; \\
\na_a \Tb^{ab}_{0,2} = -\sigma^{ac} \bar{R}_{a\;\;c\Lb}^{\;\;b} + \Tb^{ab}_{0,2} \zeta_a = -\frac{1}{2} \bar{R}_{\Lb\;\;L\Lb}^{\;\;b} + \Tb^{ab}_{0,2} \zeta_a.
\end{align}
Now, we run the same proof as Theorem \ref{minkowski1} (Theorem D) by considering
\[
\int_\Sigma \nabla_a[ T^{ab} Q(\udl L, \partial_b)]d\mu=0, \text{ and }\int_\Sigma \na_a [\Tb^{ab}Q(L,\pl_b)] d\mu =0.
\]
and get
\begin{align}
0 &= \int_\Sigma -\frac{1}{2} \bar{R}_{L\;\;\Lb L}^{\;\;b} Q(\Lb,\pl_b)	+ \mxtr\chi \langle \Lb,\frac{\pl}{\pl t} \rangle - \chi^{ab}\chib_a^c Q_{cb} - \frac{1}{2} \lt( \mxtr\chi \mxtr\chib - \chi^{ab} \chib_{ab} \rt) Q(L,\Lb)	d\mu \label{T,Lbar after Codazzi};\\
0 &= \int_\Sigma -\frac{1}{2} \bar{R}_{\Lb\;\;L\Lb}^{\;\;b} Q(L,\pl_b) + \mxtr\chib \langle L,\frac{\pl}{\pl t} \rangle -\chib^{ab}\chi_a^c Q_{cb} + \frac{1}{2} \lt( \mxtr\chi \mxtr\chib - \chib^{ab} \chi_{ab} \rt) Q(L,\Lb)	d\mu \label{Tbar,L after Codazzi}.
\end{align}

\par
Since $\frac{\partial}{\partial t} $ is a Killing field and $\vec{H}=\frac{1}{2} tr\udl\chi L+\frac{1}{2} tr\chi \udl{L}$, \[0=\int_\Sigma \langle \vec{H}, \frac{\pl}{\pl t}\rangle d\mu=\frac{1}{2} \int_\Sigma tr \udl\chi \langle L, \frac{\pl}{\pl t}\rangle
d\mu+\frac{1}{2} \int_\Sigma tr \chi \langle \udl L, \frac{\pl}{\pl t}\rangle
d\mu.\]
Subtracting (\ref{Tbar,L after Codazzi}) from (\ref{T,Lbar after Codazzi}), we obtain
\begin{align*}
&2 \int_\Sigma tr\chi \langle \udl L,\frac{\pl}{\pl t} \rangle  d\mu \\
&= \int_\Sigma \frac{1}{2} \lt( \bar{R}_{L\;\;\Lb L}^{\;\;b} Q(\Lb,\pl_b) - \bar{R}_{\Lb\;\;L\Lb}^{\;\;b} Q(L, \pl_b) \rt)+ \lt( \chi^{ab} \chib_a^c - \chib^{ab}\chi_a^c \rt)Q_{cb} + \lt( \mxtr\chi \mxtr\chib - \chi^{ab}\chib_{ab} \rt) Q(L,\Lb) d\mu \\
&= \int_\Sigma \frac{1}{2} \lt( \bar{R}_{L\;\;\Lb L}^{\;\;b} Q(\Lb,\pl_b) - \bar{R}_{\Lb\;\;L\Lb}^{\;\;b} Q(L, \pl_b) \rt) + \lt( \bar{R}_{bc\Lb L} - 2 (d\zeta)_{bc} \rt) Q^{cb} + \lt( \frac{1}{2} \bar{R}_{\Lb L L \Lb} - R \rt) Q(L,\Lb) d\mu  
\end{align*}
where we use the following Gauss and Ricci equations to get the last equality.
\begin{align}
\bar{R} + \bar{R}_{L\Lb} + \frac{1}{2} \bar{R}_{\Lb L L \Lb} = R + \mxtr\chi \mxtr\chib - \chi_{ab}\chib^{ab} \label{Gauss null};\\
\frac{1}{2} \bar{R}_{ab\Lb L} = (d\zeta)_{ab} + \frac{1}{2} \lt( \chi_a^{\;\;c} \chib_{cb} - \chib_a^{\;\;c} \chi_{cb} \rt) \label{Ricci null}.
\end{align}

\par
Next, from
\[\bar{R}_{\alpha\beta \Lb L} Q^{\alpha\beta} = \bar{R}_{bc \Lb L} Q^{bc} - \bar{R}_{L b \Lb L} Q_{\Lb}^{\;\;b} - \bar{R}_{\Lb b \Lb L} Q_L^{\;\;b} + \frac{1}{2} \bar{R}_{\Lb L \Lb L} Q(L,\Lb),\]
we have
\[ \frac{1}{2} \lt( \bar{R}_{L\;\;\Lb L}^{\;\;b} Q(\Lb,\pl_b) - \bar{R}_{\Lb\;\;L\Lb}^{\;\;b} Q(L, \pl_b) \rt) = -\frac{1}{2} \bar{R}_{\alpha\beta\Lb L} Q^{\alpha\beta} + \frac{1}{2} \bar{R}_{bc\Lb L} Q^{bc} + \frac{1}{4} \bar{R}_{\Lb L \Lb L} Q(L,\Lb), \] where $\alpha, \beta=1, \cdots, n+1$ represent the indices of the ambient space. 
Therefore,
\begin{align*}
&2 \int_\Sigma tr \chi \langle \udl L,\frac{\pl}{\pl t} \rangle d \mu \\
&= \int_\Sigma -\frac{1}{2} \bar{R}_{\alpha\beta \Lb L} Q^{\alpha\beta} + \lt( \frac{1}{2} \bar{R}_{bc \Lb L} - 2(d\zeta)_{bc} \rt)Q^{cb} + \lt( \frac{1}{4} \bar{R}_{\Lb L L \Lb} - R \rt) Q(L,\Lb) d\mu.
\end{align*}
\par
Consider the two-form	$\eta = \bar{R}_{\alpha\beta\mu\nu} Q^{\alpha\beta} dx^\mu dx^\nu$ on the spacetime. In \cite[section 3.3]{JL}, it is shown that $d \eta = d * \eta =0$. So $\int_\Sigma \bar{R}_{\alpha\beta \Lb L} Q^{\alpha\beta}$ is the same for any 2-surface bounding a 3-volume. Evaluating the integral on a sphere with $t = constant $ and $r = constant$, we have \cite[(53)]{JL}
\[ \int_\Sigma \bar{R}_{\alpha\beta \Lb L} Q^{\alpha\beta} d\mu = -32 \pi m. \]

\medskip
\par
As a result, we reach the following Minkowski formula on the 4-dimensional Schwarzschild spacetime.
\begin{theorem}[Theorem F]\label{4dimensional Schwarzschild}
Consider the two-form $Q=rdr\wedge dt$ on the 4-dimensional Schwarzschild spacetime with $m\geq 0$. For a closed oriented spacelike 2-surface $\Sigma$, we have
\begin{align*}
&2\int_\Sigma \langle \vec{H}, L\rangle \langle \udl L,\frac{\pl}{\pl t} \rangle d\mu \\\nonumber
&= - 16\pi m + \int_\Sigma \lt(R+ \frac{1}{4} \bar{R}_{L \Lb L \Lb} \rt) Q(L,\Lb) + \sum_{b,c=1}^2 \lt( \frac{1}{2} \bar{R}_{bc\Lb L} - 2(d\zeta_L)_{bc} \rt)Q_{bc} d\mu.
\end{align*}
where $\zeta_L$ is the connection 1-form of the normal bundle with respect to $L$, $\bar{R}$ is the curvature tensor of the Schwarzschild spacetime, $R$ is the scalar curvature of $\Sigma$, and $Q_{bc}=Q(e_b, e_c)$, $(d\zeta_L)_{bc}=(d\zeta_L)(e_b, e_c)$, etc.
\end{theorem}

\bigskip
\appendix
\section{Proof of Some Algebraic Relations}
In this section,  G\aa rding's theory for hyperbolic polynomials, in particular for elementary symmetric functions $\sigma_k$, is reviewed and applied to prove several algebraic relations for mixed higher order mean curvatures. For more detailed discussion about polarized $\sigma_k$ functions, we refer to the Appendix in the lecture notes by Guan \cite{Guan2004, Guan}.

\begin{definition}
Let $W^1, \cdots, W^{n-1}$ be $(n-1)\times (n-1)$ symmetric matrices, define the {\it mixed determinant} $\sigma_{(n-1)}\left(W^1, \cdots, W^{n-1}\right)$ such that $\frac{1}{(n-1)!} \sigma_{(n-1)}\left(W^1, \cdots, W^{n-1}\right)$ is the coefficient of the term $t_1\cdots t_{n-1}$ in the polynomial
\[
\det \left(t_1 W^1 +\cdots + t_{n-1} W^{n-1}\right).
\]
In general, for $1\leq k\leq n-1$, we define the {\it complete polarization} of the symmetric function $\sigma_k$ by
\begin{eqnarray}\label{polar1}
\sigma_{(k)}\left (W^1, \cdots, W^k\right) = \binom{n-1}{k}\sigma_{(n-1)} \left(W^1, \cdots, W^k, I_{n-1}, \cdots, I_{n-1}\right),
\end{eqnarray} 
where the identity matrix $I_{n-1}$ appears $(n-1-k)$ times.
\end{definition}

Higher order mixed mean curvatures can be expressed in terms of complete polarizations of the elementary symmetric functions, $\sigma_k$:
\begin{lemma}The following identity holds:
\begin{eqnarray}\label{polar2} 
P_{r, s} \left(\chi, \udl\chi\right)=\sigma_{(r+s)}(\underbrace{\chi, \cdots, \chi}_{r}, \underbrace{\udl\chi, \cdots, \udl\chi}_s).
\end{eqnarray}
\end{lemma}
\begin{proof}
 Notice that $\frac{1}{(n-1)!}\sigma_{(n-1)}( \underbrace{\chi, \cdots, \chi}_{r}, \underbrace{\udl\chi, \cdots, \udl\chi}_s, \underbrace{I_{n-1}, \cdots, I_{n-1}}_{(n-1)-(r+s)})$ is the coefficient of the term $ t_1\cdots t_{n-1}$ in the polynomial
 \[
\sigma_{n-1} \left(t_1 \chi +\cdots + t_{r} \chi + t_{r+1}\udl\chi +\cdots + t_{r+s}\udl\chi + t_{r+s+1}+\cdots + t_{n-1}\right).
\]
Denote $t=t_1+\cdots + t_r,\ \udl t = t_{r+1} + \cdots + t_{r+s},\ t_0= t_{r+s+1}+ \cdots + t_{n-1}$ and use the equation (\ref{det}), we get
\begin{align*}
&\sigma_{n-1} \left(t_1 \chi +\cdots + t_{r} \chi + t_{r+1}\udl\chi +\cdots + t_{r+s}\udl\chi + t_{r+s+1}+\cdots + t_{n-1}\right)\\
&= \det \left( t\chi + \udl t \udl\chi + t_0 I_{n-1}\right)\\
&= t_0^{n-1}\sum_{k, l}\frac{(k+l)!}{k! l!}\left(\frac{t}{t_0} \right)^k \left(\frac{\udl t}{t_0} \right)^l P_{k,l}(\chi, \udl{\chi})\\
&= \sum_{k, l}\frac{(k+l)!}{k! l!}t_0^{n-1-(k+l)}t^k \udl t^l P_{k,l}(\chi, \udl{\chi})\\
&= \sum_{k, l}\frac{(k+l)!}{k! l!}(  t_{r+s+1}+ \cdots + t_{n-1})^{n-1-(k+l)}(t_1+\cdots + t_r)^k (t_{r+1} + \cdots + t_{r+s})^l P_{k,l}(\chi, \udl{\chi}).
\end{align*}
For fixed $r, s$, the term $t_1\cdots t_{n-1}$ appears in the last expression only when $k=r$ and $l=s$, and the coefficient  is 
\[\frac{(r+s)!}{r!s!} \left( n-1-(r+s)\right)!\ r! \ s! \ P_{r, s}(\chi, \udl \chi).
\]
 Thus,
\begin{eqnarray*}
\sigma_{(n-1)}( \underbrace{\chi, \cdots, \chi}_{r}, \underbrace{\udl\chi, \cdots, \udl\chi}_s, \underbrace{I_{n-1}, \cdots, I_{n-1}}_{(n-1)-(r+s)} ) &=& \frac{(r+s)!\left( n-1-(r+s)\right)!}{(n-1)!}P_{r, s}(\chi, \udl \chi)\\
&=& \frac{1}{\binom{n-1}{r+s}}P_{r, s}(\chi, \udl\chi).
\end{eqnarray*}
(\ref{polar2}) follows from this and the definition of complete polarization \eqref{polar1}.
\end{proof}

\medskip
\par
From \eqref{polar2} and the definition of $ P_{r, s} $ and $T^{ab}_{r,s}$, we have the following basic identities:
\begin{eqnarray}
\label{b}\sum_{a, b} \chi_{ab} T_{r,s}^{ab}&=& \sum_{a, b}\chi_{ab} \frac{\partial P_{r, s}(\chi, \udl\chi)}{\partial \chi_{ab}} =rP_{r, s }(\chi, \udl\chi),\\
\label{c}\sum_{a, b} \udl\chi_{ab} \udl T_{r,s}^{ab}&=&\sum_{a, b}\udl\chi_{ab} \frac{\partial P_{r, s}(\chi, \udl\chi)}{\partial \udl\chi_{ab}} =sP_{r, s }(\chi, \udl\chi).
\end{eqnarray}
Indeed, the equalities (\ref{b}) and (\ref{c}) follow from the fact that $P_{r, s}(\chi, \udl\chi)$ is a polynomial in $\chi \text{ and } \udl\chi$, homogeneous of degrees $r$ and $s$ respectively.

In addition, the following equation $(\ref{a})$ can be deduced from the definition of complete polarized symmetric function and (\ref{polar2}):

\begin{lemma}
\begin{equation}\label{a}\sum_{a, b}\sigma_{ab} T_{r,s}^{ab}= \sum_{a, b}\sigma_{ab} \frac{\partial P_{r, s}(\chi, \udl\chi)}{\partial \chi_{ab}} = \frac{r\left(n-(r+s)\right)}{r+s}P_{r-1, s} (\chi, \udl\chi).\end{equation}
\end{lemma}
\begin{proof}[Proof of (\ref{a})]
 Indeed, we have the standard definition of completely polarized symmetric function as
\[
\sigma_{(k)}\left(W^1, W^2, \cdots, W^k \right) = \frac{1}{k!} \sum_{i_1, \cdots, i_k=1}^{n-1} W^1_{i_1} W^2_{i_2} \cdots W^k_{i_k} \frac{\partial^k \sigma_k(W)}{\partial W_{i_1}\cdots \partial W_{i_k}},
\]
where $W^j_{1}, \cdots, W^j_{n-1}$ are the eigenvalues of $W^j$. We note that $\frac{\partial^k \sigma_k(W)}{\partial W_{i_1}\cdots \partial W_{i_k}}$ is a combinatorial
constant depending
only on $k$ and $W$ can be replaced by any symmetric matrix. Thus
\begin{eqnarray}\label{polar3}
\sigma_{(k)} (\chi, \underbrace{\udl\chi, \cdots, \udl\chi}_{k-1}) &=& \frac{1}{k!}  \sum_{i_1, \cdots, i_k=1}^{n-1} \chi_{i_1} \udl\chi_{i_2}\cdots \udl\chi_{i_k} \frac{\partial^k \sigma_k(\udl\chi)}{\partial \udl\chi_{i_1}\cdots \partial \udl\chi_{i_k}}
= \frac{1}{k} \sum_{i=1}^{n-1} \chi_i \frac{\partial \sigma_k(\udl\chi)}{\partial \udl\chi_i} \\\nonumber
&=& \frac{1}{k}\ \frac{d}{d t}\Big|_{t=0}\sigma_k\left(t \chi+\udl\chi\right).
\end{eqnarray}
More generally, we have
\begin{eqnarray*}
P_{r, s} (\chi, \udl\chi) &=& \sigma_{(r+s)} ( \underbrace{\chi, \cdots, \chi}_{r},\ \underbrace{\udl\chi, \cdots, \udl\chi}_s) = \frac{s!}{(r+s)!} \ \frac{d^r}{d t^r}\Big|_{t=0}\sigma_{r+s}\left(t \chi+\udl\chi\right)\\
&=&\frac{1}{\binom{r+s}{r}}\left(\frac{1}{r!}\frac{d^r}{d t^r}\Big|_{t=0}\sigma_{r+s}(t \chi+\udl\chi)\right).
\end{eqnarray*}
Equation (\ref{a}) is verified by a sequence of direct computations:
\begin{eqnarray*}
 \sum_{a, b}\sigma_{ab} \frac{\partial P_{r, s}(\chi, \udl\chi)}{\partial \chi_{ab}} &=& \frac{1}{\binom{r+s}{r}}\frac{1}{r!}\frac{d^r}{d t^r}\Big|_{t=0}\sum_{a, b}\sigma_{ab}\frac{\partial \sigma_{r+s}(t \chi+\udl\chi)}{\partial \chi_{ab}} \\
 &=& \frac{1}{\binom{r+s}{r}}\frac{1}{r!} \frac{d^r}{d t^r}\Big|_{t=0}\left[\left(n-(r+s)\right)\sigma_{r+s-1}(t \chi+\udl\chi) t\right] \\
 &=& \frac{1}{\binom{r+s}{r}}\frac{1}{r!} \left(n-(r+s)\right) r\frac{d^{r-1}}{d t^{r-1}}\Big|_{t=0}\left[\sigma_{r+s-1}(t \chi+\udl\chi) \right]\\
 &=&\frac{n-(r+s)}{\binom{r+s}{r}} \binom{r+s-1}{r-1} P_{r-1, s}(\chi, \udl\chi)\\
 &=& \frac{\left(n-(r+s)\right) r}{ r+s}P_{r-1, s}(\chi, \udl\chi),
\end{eqnarray*}
where we use equation $\sum_{i=1}^{n-1} \frac{\partial\sigma_k(\lambda)}{\partial \lambda_i}=(n-k)\sigma_{k-1}(\lambda)$ to get the second identity above. 
\end{proof}

\medskip
\par
We now briefly review G\aa rding's inequality for polarized elementary symmetric functions and apply it to deduce a version of Newton-MacLaurin inequality for $P_{r,s}(\chi, \udl\chi)$ which is used several times in the previous sections. First, we recall the definition of the positive cone for $\sigma_k$:
\begin{definition}\label{Gamma_k cone}
For $1\leq k\leq n-1$, let $\Gamma_k$ be a cone in $\mathbb R^{n-1}$ defined by 
\[
\Gamma_k = \ \{\lambda \in \mathbb R^{n-1} \ : \ \sigma_1 (\lambda)>0, \cdots, \sigma_k(\lambda)>0 \},
\] where \[\sigma_k(\lambda) = \sum_{i_1<\cdots< i_k} \lambda_{i_1} \cdots \lambda_{i_k}\] is the $k$-th symmetric function. 
An $(n-1)\times (n-1)$ symmetric matrix $W$ is said to belong to $\Gamma_k$ if its spectrum $\lambda(W) \in \Gamma_k$. 
\end{definition}

\par
According to the standard theory for elementary symmetric functions from the hyperbolic polynomials point of view (see Corollary 13.1 and Proposition 13.3 in \cite{Guan2004}), we know that $\Gamma_k$ is the positive cone for both $\sigma_k$ and its polarization $\sigma_{(k)}$, i.e., 
\[
\sigma_{(k)}(W^1, \cdots, W^k)>0,\ \text{ for } W^i\in \Gamma_k \text{ with } i=1, \cdots, k.
\]
In view of the relation \eqref{polar2}, we also have
\begin{eqnarray}\label{positivity}
P_{r,s}(\chi, \udl \chi) >0, ,\ \text{ for } \chi, \udl\chi \in \Gamma_{r+s}.
\end{eqnarray}

\smallskip
The following lemma is a special case of a theorem of G\aa rding for hyperbolic polynomials, which can be found in \cite{Garding} (or see Appendix of \cite{Guan} or \cite{Hormander}). 

\begin{lemma}
For any $W^i\in \Gamma_k$ or $-\Gamma^k$, $i=1, \cdots, k$, we have
\begin{eqnarray*}\label{Garding}
\sigma_{(k)}^2 \left(W^1, W^2, W^3, \cdots, W^k\right) \geq \sigma_{(k)} \left(W^1, W^1, W^3, \cdots, W^k\right) \sigma_{(k)} \left(W^2, W^2, W^3, \cdots, W^k\right).
\end{eqnarray*}
The equality holds if and only if $W^1$ and $W^2$ are multiples of each other.
\end{lemma}
\begin{proof}
We may assume $W^i \in \Gamma_k$ for all $i$ because changing $W^i$ to $-W^i$ does not change the desired inequality. Since $W^3, \ldots, W^k \in \Gamma_k,$ the homogeneous polynomial $p(x) = \sigma_{(k)}(x,x,W^3, \ldots,W^k)$ is hyperbolic with respect to every element in $\Gamma_k$ (see Appendix of \cite{Guan}). The assertion follows by applying the usual G\aa rding's inequality to the complete polarization of $p$.

\medskip
\end{proof}
The above G\aa rding's inequality yields the following Newton-MacLaurin inequality for $P_{r, s}(\chi, \udl\chi) $.
\begin{lemma}
If $\chi$ and  $\udl\chi$  are both in $ \Gamma_{r+s-1}$ cone, we have
 \begin{eqnarray}\label{NM}
 P_{r-1, s}^2 (\chi, \udl\chi) \geq c(n, r, s) P_{r, s} (\chi, \udl\chi) P_{r-2, s} (\chi, \udl\chi),
 \end{eqnarray}
 where $c(n, r, s) = \frac{\binom{n-1}{r-1+s}^2}{\binom{n-1}{r-2+s}\binom{n-1}{r+s}} = \frac{r+s}{r+s-1}\cdot\frac{n-(r+s)+1}{n-(r+s)}$. The equality holds if and only if $\chi=cI_{n-1}$ for some constant $c$.
\end{lemma}
\begin{proof}
When $P_{r, s}(\chi, \udl\chi) \leq 0$, the inequality (\ref{NM})  is trivial. We thus assume $P_{r, s}(\chi, \udl\chi) \geq 0$.  Replacing $k$ by $(r+s)$ in the above G\aa rding's inequality \eqref{Garding}, we obtain
\begin{align*}
&\sigma_{(r+s)}^2 \left(W^1, W^2, W^3, \cdots, W^r, W^{r+1}, \cdots, W^{r+s}\right) \\
&\geq \sigma_{(r+s)} \left(W^1, W^1, W^3, \cdots, W^r, W^{r+1}, \cdots, W^{r+s}\right) \sigma_{(r+s)} \left(W^2, W^2, W^3, \cdots, W^r, W^{r+1}, \cdots, W^{r+s}\right).
\end{align*}

Setting $W^1= I_{n-1}, W^2=\cdots=W^r= \chi$ and $W^{r+1}=\cdots= W^{r+s} =\udl\chi$ and rewriting in terms of the complete polarization \eqref{polar1}, we derive
\begin{align*}
&\dbinom{n-1}{r+s}^2 \sigma_{(n-1)}^2 ( \underbrace{\chi, \cdots, \chi}_{r-1}, \underbrace{\udl\chi, \cdots, \udl\chi}_s, \underbrace{I_{n-1}, \cdots, I_{n-1}}_{(n-1)-(r-1+s)}   ) \\
&\geq \dbinom{n-1}{r+s}^2 \sigma_{(n-1)}( \underbrace{\chi, \cdots, \chi}_{r-2}, \underbrace{\udl\chi, \cdots, \udl\chi}_s, \underbrace{I_{n-1}, \cdots, I_{n-1}}_{(n-1)-(r-2+s)}   ) \sigma_{(n-1)}( \underbrace{\chi, \cdots, \chi}_{r}, \underbrace{\udl\chi, \cdots, \udl\chi}_s, \underbrace{I_{n-1}, \cdots, I_{n-1}}_{(n-1)-(r+s)}) 
\end{align*}
From G\aa rding's inequality, we also see that the equality holds if and only if $\chi$ and $I_{n-1}$ are proportional.
\par
On the other hand, using (\ref{polar1}) and (\ref{polar2}) again,
\begin{eqnarray*}
\sigma_{(n-1)} ( \underbrace{\chi, \cdots, \chi}_{r-1}, \underbrace{\udl\chi, \cdots, \udl\chi}_s, \underbrace{I_{n-1}, \cdots, I_{n-1}}_{(n-1)-(r-1+s)} )&=& \frac{1}{\dbinom{n-1}{r-1+s}}\sigma_{(r-1+s)}( \underbrace{\chi, \cdots, \chi}_{r-1}, \underbrace{\udl\chi, \cdots, \udl\chi}_s )\\
&=&  \frac{1}{\dbinom{n-1}{r-1+s}} P_{r-1, s}(\chi, \udl\chi),\\
\sigma_{(n-1)}( \underbrace{\chi, \cdots, \chi}_{r-2}, \underbrace{\udl\chi, \cdots, \udl\chi}_s, \underbrace{I_{n-1}, \cdots, I_{n-1}}_{(n-1)-(r-2+s)} )&=& \frac{1}{\dbinom{n-1}{r-2+s}}\sigma_{(r-2+s)}( \underbrace{\chi, \cdots, \chi}_{r-2}, \underbrace{\udl\chi, \cdots, \udl\chi}_s )\\
&=&  \frac{1}{\dbinom{n-1}{r-2+s}} P_{r-2, s}(\chi, \udl\chi),\\
\sigma_{(n-1)}( \underbrace{\chi, \cdots, \chi}_{r}, \underbrace{\udl\chi, \cdots, \udl\chi}_s, \underbrace{I_{n-1}, \cdots, I_{n-1}}_{(n-1)-(r+s)} )&=& \frac{1}{\dbinom{n-1}{r+s}}\sigma_{(r+s)}( \underbrace{\chi, \cdots, \chi}_{r}, \underbrace{\udl\chi, \cdots, \udl\chi}_s )\\
&=&  \frac{1}{\dbinom{n-1}{r+s}} P_{r, s}(\chi, \udl\chi)
\end{eqnarray*}
Thus, we reach that
\begin{eqnarray*}
P_{r-1, s}^2(\chi, \udl\chi) \geq \frac{\binom{n-1}{r-1+s}^2}{\binom{n-1}{r-2+s}\binom{n-1}{r+s}} P_{r-2, s}(\chi, \udl\chi) \cdot P_{r, s}(\chi, \udl\chi).
\end{eqnarray*}
\end{proof}

\section{The Existence of Conformal Killing-Yano forms}
In this appendix, we show the existence of conformal Killing-Yano form for a class of warped-product manifold. We recall the following equivalent definition of conformal Killing-Yano $p$-forms using the twistor equation \cite[Definition 2.1]{S}.
\begin{definition} A $p-$form $Q$ on an $n-$dimensional pseudo-Riemannian manifold $(V,g)$ is said to be a conformal Killing-Yano form if $Q$ satisfies the twistor equation
\begin{align}\label{twistor equation}
D_X Q - \frac{1}{n+1} X \lrcorner\, dQ + \frac{1}{n-p+1} g(X) \wedge d^*Q=0
\end{align}
for all tangent vector $X$.
\end{definition}
The main result of the appendix is the following existence theorem.
\begin{theorem}
Let $U \subset \mathbb{R}^n$ and $V \subset \mathbb{R}^m$ be two open sets. Let $G$ be a warped-product metric on $U \times V$ of the form
\begin{align*}
R^2(y) \sigma_{ab}(x) dx^a dx^b + g_{ij}(y) dy^i dy^j .
\end{align*}
Then $Q = R^{n+1}(y) \sqrt{\det\sigma_{ab}} dx^1 \wedge \cdots \wedge dx^n$ and $\ast Q = R(y) \sqrt{\det g_{ij}} dy^1\wedge\cdots \wedge dy^m$ are both conformal Killing-Yano forms.
\end{theorem}
\begin{proof}
By \cite[Lemma 2.3]{S}, the Hodge star-operator $\ast$ maps conformal Killing-Yano $p-$form into conformal Killing-Yano $(n+m-p)-$form. It suffices to verify that $Q$ satisfies the twistor equation. Let $\omega^\alpha, \alpha = 1, \ldots, n+m$ be a local orthonormal coframe for $G$ such that $\omega^1, \ldots, \omega^n$  is an orthonormal coframe for $R^2(y) \sigma_{ab}(x) dx^a dx^b$ on each slice $U \times \{c_1\}$ and $\omega^{n+1},\ldots, \omega^{n+m}$ is an orthonormal coframe for $g_{ij}(y) dy^i \wedge dy^j$ on each slice $\{c_2\} \times V$. Let $E_\alpha$ be the dual frame to $\omega^\alpha.$ If we write $\Omega = \omega^1 \wedge \cdots \wedge \omega^n,$ then $Q = R \Omega.$ From the structure equations
\begin{align*}
d\omega^a &= -\omega^a_{\;\;b} \wedge \omega^b - \omega^a_{\;\;n+i} \wedge \omega^{n+i} = dR\wedge\sigma^a - R \gamma^a_{\;\;b}\wedge \sigma^b \\
d\omega^{n+i} &= -\omega^{n+i}_{\;\;b} \wedge \omega^b - \omega^{n+i}_{\;\;n+j} \wedge \omega^{n+j},
\end{align*}
we solve for the connection 1-forms
\begin{align*}
\omega^a_{\;\;n+i} = \frac{E_{n+i}(R)}{R} \omega^a, \quad \omega^a_{\;\;b} = \gamma^a_{\;\;b}
\end{align*}
where $\gamma^a_{\;\;b}$ are the connection 1-forms with respect to the metric $\sigma_{ab}(x) dx^a dx^b.$ 

We compute each term in the twistor equation.
\begin{align*}
D_X Q = X(R) \Omega + R \na_X \Omega = X(R) \Omega - \sum_{i=1}^m E_{n+i}(R) \omega^{n+i} \wedge (X \lrcorner\,  \Omega).
\end{align*}
\begin{align*}
X \lrcorner\, d\Omega &= X \lrcorner\,  \sum_{i=1}^m \left( -\omega^1_{\;\;n+i} \wedge\omega^{n+i} \wedge \omega^2 \wedge \cdots \wedge \omega^n + \omega^1 \wedge\omega^2_{\;\;i} \wedge \omega^{n+i} \wedge \cdots \wedge \omega^n - \cdots \right)\\
&= X \lrcorner\, \sum_{i=1}^m n \left( \frac{E_{n+i}(R)}{R} \omega^{n+i} \wedge \Omega \right)\\
&= n \sum_{i=1}^m \frac{E_{n+i}(R)}{R} X \lrcorner\, (\omega^{n+i} \wedge \Omega)\\
&= n \sum_{i=1}^m \frac{E_{n+i}(R)}{R} \left( \omega^{n+i}(X) \Omega - \omega^{n+i} \wedge	(X \lrcorner\, \Omega) \right)
\end{align*}
This implies that
\begin{align*}
X \lrcorner\, dQ &=X \lrcorner\, (dR \wedge \Omega + R d\Omega) \\
&= X(R) \Omega - dR \wedge (X \lrcorner\, \Omega) + n \sum_{i=1}^m E_{n+i}(R) \left( \omega^{n+i}(X) \Omega - \omega^{n+i} \wedge (X \lrcorner\, \Omega) \right).
\end{align*}
On the other hand, $d^* Q =0$ since $R$ only depends on $y$. 
Putting these facts together, we verify that $Q$ satisfies the twistor equation
\begin{align*}
& D_X Q - \frac{1}{n+1} X \lrcorner\, dQ + \frac{1}{m+1} g(X) \wedge d^*Q\\
&= X(R) \Omega - \sum_{i=1}^m E_{n+i}(R) \omega^{n+i} \wedge (X \lrcorner\, \Omega) \\
&\quad - \frac{1}{n+1} \Big( X(R) \Omega - dR \wedge (X \lrcorner\, \Omega) -  n \sum_{i=1}^m E_{n+i}(R)\left( \omega^{n+i}(X) \Omega - \omega^{n+i} X \lrcorner\, \Omega \right)  \Big)\\
&= \frac{n}{n+1} X(R) \Omega - \frac{1}{n+1} E_{n+i}(R) \omega^{n+i} \wedge (X \lrcorner\, \Omega) \\
&\quad + \frac{1}{n+1} dR \wedge(X \lrcorner\, \Omega) - \frac{n}{n+1} E_{n+i}(R) \omega^{n+i}(X) \Omega \\
&= 0.
\end{align*}
We use the fact that $R$ only depends on $y$ in the last equality.
\end{proof}
We have the following existence result, generalizing the fact that $r dr\wedge dt$ is a conformal Killing-Yano two-form on the Minkowski and Schwarzschild spacetime.

\begin{corollary}\label{existence of CKY in warped product}
Let $(V,g)$ be a warped product manifold with
\begin{align}
g = g_{tt}(t,r)dt^2 + 2g_{tr}(t,r)dt dr + g_{rr}(t,r)dr^2 + r^2(g_N)_{ab}dx^a dx^b
\end{align}
where $(N,g_N)$ is an $(n-1)-$dimensional Riemannian manifold. Then the two-form \begin{align*}
Q = r \sqrt{\left|\det \left( \begin{array}{cc}
g_{tt}&g_{tr}\\
g_{rt}&g_{rr}
\end{array} \right)\right|} \;\;dr \wedge dt
\end{align*}
is a conformal Killing-Yano two-form on $(V,g).$
\end{corollary}

\section{Curvature tensors of the Schwarzschild spacetime}
We consider the $(n+1)$-dimensional (exterior) Schwarzschild spacetime with the metric  \eqref{Schwarzschild}. The spacetime admits a conformal Killing-Yano tensor $Q = r dr \wedge dt$.
Let $Q^2$ be the symmetric 2-tensor given by
\begin{equation}\label{Q^2} (Q^2)_{\alpha\beta} = Q_\alpha^{\;\;\gamma} Q_{\gamma\beta}.\end{equation}
\begin{lemma}\label{curvature tensor}
The curvature tensor of Schwarzschild spacetime can be expressed as
\begin{align}
\begin{split}
\bar{R}_{\alpha\beta\gamma\delta} &= \frac{2m}{r^n} \lt( \bar{g}_{\alpha\gamma} \bar{g}_{\beta\delta} - \bar{g}_{\alpha\delta} \bar{g}_{\beta\gamma} \rt) - \frac{n(n-1)m}{r^{n+2}} \lt( \frac{2}{3}	Q_{\alpha\beta} Q_{\gamma\delta} - \frac{1}{3} Q_{\alpha\gamma}Q_{\delta\beta}-\frac{1}{3}Q_{\alpha\delta}Q_{\beta\gamma} \rt) \\
&\quad - \frac{nm}{r^{n+2}} \Big( \bar{g}\circ Q^2 \Big)_{\alpha\beta\gamma\delta} \label{curvature tensor in terms of Q: Schwarzschild}
\end{split}
\end{align}
where $(\bar{g}\circ Q^2)_{\alpha\beta\gamma\delta} = \bar{g}_{\alpha\gamma} (Q^2)_{\beta\delta} - \bar{g}_{\alpha\delta} (Q^2)_{\beta\gamma} + \bar{g}_{\beta\delta} (Q^2)_{\alpha\gamma} - \bar{g}_{\beta\gamma} (Q^2)_{\alpha\delta}$ 
\end{lemma}

\begin{proof}
Denote $f^2 = 1-\frac{2m}{r^{n-2}}$. Let $E_1, E_2, \ldots, E_{n+1}$ be the orthonormal frames for $\bar{g}$ with $E_{n+1} = \frac{1}{f} \frac{\pl}{\pl t}, E_n = f \frac{\pl}{\pl r}$ and $E_i, i=1, \ldots, n-1$ tangent to the sphere of symmetry. We have
\begin{align*}
\bar{R}(E_{n+1},E_n, E_{n+1}, E_n) &= -\frac{m(n-1)(n-2)}{r^n} \\
\bar{R}(E_{n+1},E_i, E_{n+1},E_j) &= \frac{m(n-2)}{r^n} \delta_{ij} \\
\bar{R}(E_n,E_i,E_n,E_j) &= -\frac{m(n-2)}{r^n} \delta_{ij} \\
\bar{R}(E_i,E_j,E_k,E_l) &= \frac{2m}{r^n} (\delta_{ik}\delta_{jl} - \delta_{il}\delta_{jk})
\end{align*} 
Except for the symmetries of the curvature tensors, the other components are zero.

On the other hand, we have $Q(E_n,E_{n+1}) = r, (Q^2)(E_{n+1},E_{n+1}) = -r^2,$ and $(Q^2)(E_n,E_n) = r^2$. Let $b(Q) = \frac{2}{3}	Q_{\alpha\beta} Q_{\gamma\delta} - \frac{1}{3} Q_{\alpha\gamma}Q_{\delta\beta}-\frac{1}{3}Q_{\alpha\delta}Q_{\beta\gamma}$. The following table lists the nonzero components for the $(0,4)$-tensors involved.
\begin{align*}
\begin{array}{c|ccc}
T & \bar{g}_{\alpha\gamma} \bar{g}_{\beta\delta} - \bar{g}_{\alpha\delta} \bar{g}_{\beta\gamma}  & b(Q)& (\bar{g} \circ Q^2)_{\alpha\beta\gamma\delta}\\
\hline
T(E_{n+1},E_n,E_{n+1},E_n) & -1 & r^2 & -2r^2\\
T(E_{n+1},E_i,E_{n+1},E_j) & -\delta_{ij} & 0 & -r^2 \delta_{ij}\\
T(E_n,E_i,E_n,E_j) & \delta_{ij} & 0 & r^2 \delta_{ij}\\
T(E_i,E_j,E_k,E_l) & \delta_{ik}\delta_{jl}-\delta_{il}\delta_{jk} &0 &0
\end{array}
\end{align*}
Suppose $\bar{R}_{\alpha\beta\gamma\delta} = A \frac{m}{r^n} \lt( \bar{g}_{\alpha\gamma}\bar{g}_{\beta\delta} - \bar{g}_{\alpha\delta}\bar{g}_{\beta\gamma} \rt) + B \frac{m}{r^{n+2}} \lt( \frac{2}{3}	Q_{\alpha\beta} Q_{\gamma\delta} - \frac{1}{3} Q_{\alpha\gamma}Q_{\delta\beta}-\frac{1}{3}Q_{\alpha\delta}Q_{\beta\gamma} \rt) + C \frac{m}{r^{n+2}} (\bar{g} \circ Q^2)_{\alpha\beta\gamma\delta}$. We can solve for $A=2, B=-n(n-1),$ and $C=-n$.
\end{proof}
  
\bigskip
{
\end{document}